\documentclass{amsart}
\usepackage[utf8]{inputenc}

\usepackage{amsthm,amssymb,latexsym,amsmath}
\usepackage[all]{xy}
\usepackage{tikz}
\usepackage{color}
\usepackage[english]{babel}
\usepackage{spectralsequences}
\usepackage{tikz-cd}
\usepackage{tikz}
\usepackage{enumitem}
\newcommand{\Z}{{\mathbb Z}}
\newcommand{\C}{\mathbb{C}}

\newcommand{\p}[1]{{\mathbb{P}^{#1}}}

\newcommand{\op}[1]{{\mathcal O}_{\mathbb{P}^{#1}}}

\newcommand{\ox}{{\mathcal O}_{X}}

\newcommand{\OC}{{\mathcal O}_{C}} 

\newcommand{\OO}{{\mathcal O}}

\newcommand{\ch}{\operatorname{ch}}

\newcommand{\sing}{\operatorname{Sing}}
\newcommand{\supp}{\operatorname{Supp}}

\newcommand{\Grp}{\mathbb{G}}
\newcommand{\cala}{{\mathcal A}}
\newcommand{\calb}{{\mathcal B}}
\newcommand{\calc}{{\mathcal C}}
\newcommand{\cald}{{\mathcal D}}
\newcommand{\cale}{{\mathcal E}}
\newcommand{\calf}{{\mathcal F}}
\newcommand{\calg}{{\mathcal G}}
\newcommand{\calh}{{\mathcal H}}
\newcommand{\cali}{{\mathcal I}}

\newcommand{\call}{{\mathcal L}}
\newcommand{\calm}{{\mathcal M}}
\newcommand{\caln}{{\mathcal N}}
\newcommand{\calo}{{\mathcal O}}
\newcommand{\calp}{{\mathcal P}}

\newcommand{\calr}{{\mathcal R}}
\newcommand{\cals}{{\mathcal S}}
\newcommand{\Ss}{{\mathcal S}}
\newcommand{\calt}{{\mathcal T}}

\newcommand{\calz}{{\mathcal Z}}
\newcommand{\bff}{{\mathbf F}}
\newcommand{\bfe}{{\mathbf E}}
\newcommand{\bfw}{{\mathbf W}}
\newcommand{\bfl}{{\mathbf l}}
\newcommand{\inhom}{{\mathcal H}{\it om}}
\newcommand{\inext}{{\mathcal E}{\it xt}}
\newcommand{\intor}{{\mathcal T\!}{\it or}}
\newcommand{\Ext}{\operatorname{Ext}}
\newcommand{\ext}{\operatorname{ext}}
\newcommand{\End}{\operatorname{End}}
\newcommand{\Sym}{\operatorname{Sym}}
\newcommand{\spec}{\operatorname{Spec}}
\newcommand{\Tan}{\operatorname{Tan}}
\newcommand{\Hom}{\operatorname{Hom}}

\newcommand{\Hilb}{\operatorname{Hilb}}

\DeclareMathOperator{\coker}{coker}
\DeclareMathOperator{\im}{im}

\DeclareMathOperator{\rk}{{rk}}

\DeclareMathOperator{\Pic}{{Pic}}
\DeclareMathOperator{\Gr}{{Gr}}

\newcommand{\into}{\hookrightarrow}
\newcommand{\onto}{\twoheadrightarrow}

\newtheorem{theorem}{Theorem}

\newtheorem{proposition}[theorem]{Proposition}

\newtheorem{lemma}[theorem]{Lemma}
\newtheorem{corollary}[theorem]{Corollary}

\theoremstyle{definition}
\newtheorem{remark}[theorem]{Remark}
\newtheorem{example}[theorem]{Example}
\newtheorem{definition}[theorem]{{\bf Definition}}
\newcommand{\D}[1]{{\mathbb#1}}
\newcommand{\calu}{{\mathcal U}}
\begin{document}

\title[Instanton sheaves on Fano threefolds]{Instanton sheaves on Fano threefolds}

\author{Gaia Comaschi}
\author{Marcos Jardim}
\address{IMECC - UNICAMP \\ Departamento de Matem\'atica \\
Rua S\'ergio Buarque de Holanda, 651\\ 13083-970 Campinas-SP, Brazil}
\email{jardim@ime.unicamp.br}

\begin{abstract}
Generalizing the definitions originally presented by Kuznetsov and Faenzi, we study (possibly non locally free) instanton sheaves of arbitrary rank on Fano threefolds. We classify rank 1 instanton sheaves and describe all curves whose structure sheaves are rank 0 instanton sheaves. In addition, we show that every rank 2 instanton sheaf is an elementary transformation of a locally free instanton sheaf along a rank 0 instanton sheaf. To complete the paper, we describe the moduli space of rank 2 instanton sheaves of charge 2 on a quadric threefold $X$, and show that the full moduli space of rank 2 semistable sheaves on $X$ with Chern classes $(c_1,c_2,c_3)=(-1,2,0)$ is connected and contains, besides the instanton component, just one other irreducible component which is also fully described.
\end{abstract}

\maketitle

\tableofcontents


\section{Introduction}

In their seminal work \cite{ADHM}, Atiyah, Drinfeld, Hitchin and Manin presented the notion of \textit{mathematical instantons}, rank 2 holomorphic vector bundles on $\p3$ that correspond to anti-self-dual connections, a.k.a. instantons, on the sphere $S^4$. More precisely,  a mathematical instanton of charge $n$ is defined as a stable rank 2 vector bundle $E$ with Chern classes $c_1(E)=0, \ c_2(E)=n$ and such that $H^i(E(-2))=0$ for $i=1,2$.

In the following years several authors presented different generalizations of mathematical instanton first to odd dimensional projective spaces \cite{MCS}, then to non locally free sheaves of any rank on arbitrary projective spaces \cite{J-inst}, and to other Fano threefolds besides $\p3$ \cite{F,K,CCGM}, and more recently to arbitrary polarized projective varieties \cite{AC,AM}.

In the present paper, we introduce a definition of instanton sheaves on a Fano threefold $X$ of Picard rank one, compatible with all the aforementioned ones. Namely, a torsion free sheaf $E$ on $X$ is said to be an \textit{instanton sheaf} of charge $n$ if it is $\mu$-semistable and it satisfies $c_1(E)=-r_X, \ c_2(E)=n$ and $h^i(E(-q_X))=0$ for $i=1,2$, with $r_X, \ q_X$ integers such that $K_X\sim -(2q_X+r_X)H_X$ where $K_X$ and $H_X$ are, respectively the anticanonical class and the ample generator of $\Pic(X)$. These requirements are nevertheless sufficient to guarantee that several expected properties for instanton sheaves still hold in this more general setting. 

Once the notion of an instanton sheaf has been provided and their main features have been illustrated, we focus our attention on how they behave in families with a particular emphasis on the rank 2 case. Indeed, there exists a vast literature on moduli spaces of rank 2 instanton sheaves. 

On $\p3$, the moduli space $\cali(n)$ of rank 2 instanton bundles on $\p3$ has been proved to be an irreducible \cite{T1,T2} smooth \cite{JV} affine \cite{CO} variety of dimension $8n-3$. A more comprehensive picture of $\overline{\cali(n)}$, the closure of $\cali(n)$ within the Gieseker--Maruyama moduli scheme $\calm_{\p3}(0,n,0)$ of rank 2 semistable sheaves on $\p3$ with Chern classes $c_1=0,\ c_2=n,\ c_3=0$, can then be obtained taking into account also the non locally free instanton sheaves, as shown in \cite{JMT1,JMT2}. Moreover, the moduli space $\call(n)$ of all rank 2 instanton sheaves of charge $n$ was shown in \cite{JMaiT} to be connected for $n\le4$.

Moduli spaces of rank 2 instanton bundles on other Fano threefolds $X$ have been widely inspected as well: among them we mention for example \cite{Druel,dima-tikho,Ott-Szu}.
In these works, a frequently used technique to construct instantons and study their moduli is the so called \textit{Serre correspondence}. This latter establishes a correspondence between global sections of rank 2 reflexive sheaves on $X$ and locally Cohen Macaulay (l.c.m.) curves on $X$; it is then possible to deduce geometrical properties of moduli spaces of sheaves from the geometry of the Hilbert scheme of curves on $X$.

In order to carry on these pursuits, we present in this paper a more general form of the Serre correspondence that applies to torsion free sheaves. This will allow us to describe in greater detail the geometry of the curves corresponding to the non locally free rank 2 instanton sheaves, and then apply these results to study moduli spaces of rank 2 instanton sheaves on the quadric threefold. 

Here is the plan of the paper. In Section \ref{sec:fano} we set up the notation for the rest of the paper by recalling the classification of Fano threefolds of Picard rank one, and some features of (semi)stable sheaves. Section \ref{sec:serre} is dedicated to formulating the Serre correspondence for torsion free sheaves, generalizing the classical correspondence presented in \cite{Arrondo,H-bundle,H-reflexive}.

Instanton sheaves on Fano threefolds, the main characters of our tale, are then introduced in Section \ref{sec:instantons}. After going over some basic properties and examples of instanton sheaves, we also introduce the notion of \textit{rank 0 instanton sheaves}, that is 1-dimensional sheaves $T$ on $X$ satisfying the cohomological vanishing $H^i(T(-q_X))=0, \ i=0,1$; for $X=\p3$, rank 0 instantons were originally introduced in \cite{HLa} and further studied in \cite{GJ}, and play a key role in the study of non-reflexive instanton sheaves via a procedure known as \textit{elementary transformation}.
We present a classification of the rank 0 instantons $T$ of the form $T\simeq \calo_C$ where $C\subset X$ a l.c.m. curve (that will be therefore referred to as an \textit{instanton curve}) and as a direct consequence of this result we obtain a classification of rank 1 instanton sheaves, see Proposition \ref{rk1}.

The first main result of the paper is a classification of rank 2 instanton sheaves, see Theorem \ref{double-dual-bundle}. To be precise, we prove that in this case an instanton sheaf $E$ is either locally free or its singular locus has pure dimension one. This implies in particular that every non locally free rank 2 instanton sheaf $E$ is obtained via an elementary transformation of a locally free instanton sheaf $F$ along a rank 0 instanton sheaf $T$; if this occur we moreover have $E^{\vee\vee}\simeq F$ and $\sing(E)=\supp(T)$. 

We complete Section \ref{sec:instantons} with a detailed description of the Serre correspondence for non locally free rank 2 instanton sheaves. In particular, we describe how curves corresponding to locally free instanton sheaves behave under elementary transformation; in this way we can also relate the geometry of the curves corresponding to non locally free instanton sheaves to the singularities of these sheaves.

This formulation of the Serre correspondence is the main tool used in Section \ref{sec:quadrics} to study the moduli space of instantons of charge $2$ on the quadric threefold $X\subset \p4$. The Serre correspondence was used in \cite{Ott-Szu} to describe the moduli space $\cali(2)$ of instanton bundles of charge 2 on $X$ and to prove that the latter is an irreducible smooth variety of dimension 6. Our study of the families of curves corresponding to non locally free instanton sheaves allows us to prove that these always deform to locally free instanton sheaves, and that they are parameterized by an irreducible divisor in $\overline{\cali(2)}$.

Still relying on the Serre correspondence, we give a complete description of the moduli space $\calm_X(2;-1,2,0)$ of rank 2 semistable sheaves with Chern classes $c_1=-1, \ c_2=2, \ c_3=0$ in Section \ref{sec:moduli}. We prove that this moduli space consists of two irreducible components, namely $\overline{\cali(2)}$, and a 10-dimensional irreducible component $\calc$ whose general point is the kernel of an epimorphism $F\twoheadrightarrow \calo_p$, with $F$ a stable reflexive sheaf with $c_1(F)=-1, \ c_2(F)=2, \ c_3(F)=2$ and $p\in X$ a point. We will finally show that $\overline{\cali(2)}\cap \calc\ne \emptyset$ proving the connectedness of $\calm_X(2;-1,2,0)$.

\subsection*{Acknowledgments}
MJ is supported by the CNPQ grant number 302889/2018-3 and the FAPESP Thematic Project 2018/21391-1. GC is supported by the FAPESP grant number 2019/21140-1. 


\section{Background and notation}\label{sec:fano}

\subsection{Classification of Fano threefolds} 

Let $X$ be a smooth 3 dimensional projective variety whose Picard group has rank one. Letting $H_X$ denote the ample generator of $\Pic(X)$, the canonical class $K_X$ can be written in the form
$$ K_X=-i_XH_X \hspace{1cm} i_X\in \D{Z};$$
$X$ is said to be \textit{Fano} whenever $i_X>0$.
For each Fano threefold $X$, we define the following numerical invariants:
\begin{itemize}
    \item the \textbf{index}, defined as the integer $i_X$;
    \item the \textbf{degree} $d_X:=H_X^3$;
\end{itemize}
In addition, we let $q_X$ and $r_X$ denote the quotient and the remainder of the division of $i_X$ by 2, so that we can write $i_X=2q_X+r_X,$ with $q_X\ge 0$ and $ r_X\in\{0,1\}$.

The cohomology groups of a Fano threefold $X$ satisfies the following properties.
All the groups $H^{i,i}(X)$ have dimension one, and for this reason towards the entire article we will write the Chern classes of any sheaf $F\in Coh(X)$ as integers.
By Kodaira vanishing we then compute:
\begin{align*}
    h^i(\calo_X(k))&=0,\ i=1,2, \ k\in \D{Z}\\
    h^{0,p}(X)&=h^{p,0}(X)=0.
\end{align*}

Fano threefolds with Picard rank one were classified by Iskovskikh and Mukai \cite{Isk-Fano,Mukai}. Recall that $i_X\le4$, and 
\begin{itemize}
    \item if $i_X=4$, then $X\simeq \D{P}^3$;
    \item if $i_X=3$, then $X$ is a smooth quadric hypersurface in $\D{P}^4$.
    \item there are five families of Fano threefolds with $i_X=2$, up to deformation; these families are classified according to its degree $d_X\in\{1,2,3,4,5\}$.
    \item there are ten deformation families of Fano threefolds with $i_X=1$, which are also classified according to its degree $d_X$ taking all even values between 2 and 22, except 18.
\end{itemize}
Note that even if in some cases we have different isomorphism classes of Fano threefolds, these belongs to the same deformation family.


\subsection{Stability of sheaves}\label{sec:stability}

Let $E$ be a coherent sheaf on a non-singular projective variety $X$ with $\Pic(X)=\Z$; let $E(t):=E\otimes H_X^{\otimes t}$, where $H_X$ denotes the ample generator of $\Pic(X)$. 

Recall that $E$ is \emph{(semi)stable} if for every non-trivial subsheaf $F\subset E$ we have
$$ p_F(t) < ~(\le)~ p_E(t) , $$
where $p_E(t)$ denotes the reduced Hilbert polynomial of the sheaf $E$. Furthermore, when $E$ is a torsion free sheaf, $E$ is \emph{$\mu$-(semi)stable} if for every non-trivial subsheaf $F\subset E$ such that $E/F$ is also torsion free we have
$$ \dfrac{c_1(F)\cdot H_X^{d-1}}{\rk(F)} < ~(\le)~ \dfrac{c_1(E)\cdot H_X^{d-1}}{\rk(E)} . $$
Remark that, for torsion free sheaves:
$$ \mu-{\rm stability} ~~ \Longrightarrow ~~ {\rm stability} ~~ \Longrightarrow ~~ {\rm semistability} ~~ \Longrightarrow ~~ \mu-{\rm semistability}, $$
see \cite[Lemma 1.2.13]{HL}.
In addition, $E$ is $\mu$-(semi)stable if and only if $E^\vee$ is $\mu$-(semi)stable.

Here is a simple characterization of $\mu$-(semi)stable rank 2 sheaves, which generalizes well known results for reflexive sheaves, cf. \cite[Lemma II.1.2.5]{OSS}. Recall that a torsion free sheaf $E$ is said to be \emph{normalized} if $c_1(E)\in\{0,-1,\dots,-\rk(E)+1\}$; every torsion free sheaf can be normalized after a twist by $\ox(k)$ for some suitable integer $k$. Recall also that a normalized torsion free sheaf with $c_1(E)<0$ is $\mu$-stable if and only if it is $\mu$-semistable. 

\begin{lemma}\label{lem:stable rk 2}
Let $E$ be a normalized torsion free sheaf. 
\begin{itemize}
\item[(1)] Assuming $c_1(E)=0$, we have that 
\begin{itemize}
\item[(1.1)] if $E$ is $\mu$-stable then $h^3(E\otimes\omega_X)=0$;
\item[(1.2)] if $E$ is $\mu$-semistable then $h^3(E\otimes\omega_X(1))=0$;
\item[(1.3)] the converse statements hold when $\rk(E)=2$.
\end{itemize}
\item[(2)] Assuming $c_1(E)<0$, we have that 
\begin{itemize}
\item[(2.1)] if $E$ is $\mu$-semistable then $h^3(E\otimes\omega_X)=0$;
\item[(2.2)] the converse statement holds when $\rk(E)=2$.
\end{itemize}
\end{itemize}
\end{lemma}

\begin{proof}
By Serre duality, we know that $H^3(E\otimes\omega_X)\simeq\Hom(E,\ox)^*$. If $h^3(E\otimes\omega_X)>0$, then there is a non-trivial morphism $\varphi:E\to\ox$; let $F:=\ker(\varphi)$. Since $c_1(F)=-c_1(\im(\varphi))\ge0$, we conclude that $E$ cannot be $\mu$-stable.

Conversely, assume that $\rk(E)=2$; if $E$ is not $\mu$-stable, let $F$ be a destabilizing subsheaf; since $E$ has rank 2, we must have that both $F$ and $G:=E/F$ are rank 1 torsion free sheaves. It follows that $G=\cali_\Gamma(k)$ for some subscheme $\Gamma\subset X$ and $k\le0$, thus there exists a monomorphism $G\into \ox$; composing the epimorphism $E\onto G$ with the latter, we obtain a non trivial morphism $E\to\ox$, showing that  $h^3(E\otimes\omega_X)>0$.

The proofs for item (1.2) and (2) are completely analogous.
\end{proof}

We will need the following result regarding 1-dimensional sheaves.

\begin{lemma}\label{lem:T-sst}
Let $T$ be a pure 1-dimensional sheaf with $\chi(T(t))=d\cdot(t+e)$. If $h^0(T(-e))=0$, then $T$ is semistable.
\end{lemma}
\begin{proof}
If $S\into T$ is a subsheaf, then $h^0(S(-e))=0$; if we set $\chi(S(t))=s\cdot t + x$, then $\chi(S(-e))=-se+x=-h^1(L(-1))\le0$, thus $x\le se$; note that $s>0$ because $T$ has pure dimension 1. It follows that 
$$ \dfrac{\chi(T(t))}{d} - \dfrac{\chi(S(t))}{s} = e - \dfrac{x}{s} \ge 0, $$
proving that $T$ is semistable.
\end{proof}


\section{Serre correspondence for torsion free sheaves}\label{sec:serre}

The so-called \emph{Serre correspondence} is  one of the most efficient tools to construct and study rank 2 sheaves on a threefold $X$.

Recall from \cite[Theorem 4.1]{H-reflexive} (which generalizes \cite[Theorem 1.1]{H-bundle}) that this is a correspondence relating pairs $(C, \xi)$ consisting of a curve $C$ in $X$ and a global section $\xi$ of a twist of the dualizing sheaf $\omega_C:=\inext^2(\calo_C,\omega_X)$, with pairs $(E,s)$ consisting of a rank 2 reflexive sheaf $E$ and a global section $s\in H^0(E)$ whose cokernel is torsion-free. Another version of the Serre correspondence was given by Arrondo in \cite[Theorem 1.1]{Arrondo}, including locally free sheaves of higher rank.

The main goal of this section is to consider a generalization of the Serre correspondence for torsion free sheaves on projective varieties generalizing all of the results mentioned above.

\begin{theorem} \label{thm:serre}
Let $X$ be a non-singular, projective variety and let $L$ be a line bundle on $X$ such that $H^1(L^\vee)=H^2(L^\vee)=0$. There is a correspondence between
\begin{itemize}
\item sets $(E,s_1,\cdots,s_{r-1})$ consisting of a rank r torsion free sheaf $E$ with $\det(E)=L$, and global sections $s_1,\dots,s_{r-1}\in H^0(E)$ whose dependency locus has codimension at least 2;
\item sets $(C,\xi_1,\dots,\xi_{r-1})$ consisting of a codimension 2 subscheme $C\subset X$ and sections $\xi_1,\dots,\xi_{r-1}\in H^0(\omega_C\otimes\omega_X^{-1}\otimes L^{-1})$.
\end{itemize}
\end{theorem}

\begin{proof}
Starting with a set $(E,s_1,\cdots,s_{r-1})$ as described in the first item, we form a morphism
$$ \sigma:=(s_1,\cdots,s_{r-1}) : \ox^{\oplus(r-1)} \longrightarrow E ; $$
the hypothesis on $(s_1,\cdots,s_{r-1})$, which means that the common zeros of $s_i$ have codimension at least 2, imply that $\sigma$ is surjective and $\coker(\sigma)$ is a torsion free sheaf. It follows that $\coker(s)\simeq\cali_C\otimes L$ where $C$ is a subscheme of codimension 2, which is precisely the dependency locus of $(s_1,\cdots,s_{r-1})$, and $L$ is a line bundle. Therefore, we obtain a short exact sequence of the form
\begin{equation} \label{sc-tf}
0 \longrightarrow \ox^{\oplus(r-1)} \stackrel{\sigma}{\longrightarrow} E \longrightarrow \cali_C\otimes L \longrightarrow 0;
\end{equation}
in addition, this exact sequence defines an extension class
$$ \xi\in\Ext^1(\cali_C\otimes L,\ox^{\oplus(r-1)}). $$

Using the spectral sequence for local-to-global Ext's
\begin{equation} \label{ss-loc-glob-ext}
    H^p(\inext^q(\cali_C,\ox))\Rightarrow \Ext^{p+q}(\cali_C,\ox)
\end{equation}
one checks that the hypothesis $H^1(L)=H^2(L)=0$ yields the first of the following isomorphisms
\begin{equation} \label{isos}
\Ext^1(\cali_C\otimes L,\ox) \simeq H^0(\inext^1(\cali_C\otimes L,\ox)) \simeq H^0(\inext^2(\calo_C,\omega_X)\otimes\omega_X^{-1}\otimes L^{-1}).    
\end{equation}
This means that the extension class $\xi$ can be regarded as $r-1$ section $\xi_1,\dots,\xi_{r-1}\in H^0(\omega_C\otimes\omega_X^{-1}\otimes L^{-1})$. We have thus obtained a set $(C,\xi_1,\dots,\xi_{r-1})$ as described in the second item.

Conversely, given a set $(C,\xi_1,\dots,\xi_{r-1})$ we can use the isomorphisms in display \eqref{isos} to re-interpret $\xi_1,\dots,\xi_{r-1}$ as an extension class in $\Ext^1(\cali_C\otimes L,\ox^{\oplus(r-1)})$ leading to an exact sequence as in display \eqref{sc-tf}, which yields a set $(E,s_1,\cdots,s_{r-1})$. 
\end{proof}

In general, the abelian groups $\Ext^1(\cali_C,L)$ and $H^0(\inext^1(\cali_C,L))$ are related via the following exact sequence
$$ 0 \longrightarrow H^1(L^\vee) \longrightarrow \Ext^1(\cali_C\otimes L,\ox) \longrightarrow H^0(\inext^1(\cali_C\otimes L,\ox)) \longrightarrow H^2(L^\vee); $$
here, we used the isomorphism $\inhom(\cali_C\otimes L,\ox)\simeq L^\vee$. Therefore, if one only assumed that $H^2(L^\vee)=0$, then every set $(C,\xi_1,\dots,\xi_{r-1})$ defines an extension class in $\Ext^1(\cali_C\otimes L,\ox^{\oplus(r-1)})$ and thus a torsion free sheaf of rank $r$. 

In this paper, we will only be concerned with threefolds, so we fix $\dim(X)=3$ once and for all. Moreover, we will mostly consider only rank 2 sheaves.

\begin{remark} \label{rem:five}
Fix $r=2$.
\begin{enumerate}
\item $E$ is reflexive if and only if the scheme $C$ is locally Cohen--Macaulay (l.c.m.) and $\xi:\ox\to\inext^2(\OC,\ox)\otimes L^{-1}$ only vanishes on a 0-dimensional subscheme $Z\subset C$. In addition, $Z$ coincides with the singular locus of $E$.
\item $E$ is locally free if and only if the scheme $C$ is locally complete intersection (l.c.i.) and $\xi:\ox\to\omega_C\otimes\omega_X^{-1}\otimes L^{-1}$ is non vanishing.
\end{enumerate}
Detailed explanation for these claims can be found in the classical references \cite[Theorem 4.1]{H-reflexive} and \cite[Theorem 1.1]{H-bundle}, respectively.
\end{remark}

When $E$ is not reflexive, set $T_E:=E^{\vee\vee}/E$ and consider the following commutative diagram
\begin{equation}\begin{split}\label{snake-global-sect}
\xymatrix{
& 0\ar[d] & 0\ar[d] & & \\
& \ox\ar[d]^{s}\ar@{=}[r] & \ox\ar[d]^{\iota(s)} & & \\
0\ar[r] & E \ar[r]\ar[d] & E^{\vee\vee} \ar[r]^{q}\ar[d] & T_E \ar[r]\ar@{=}[d] & 0 \\
0\ar[r] & \cali_{C}\otimes L \ar[r]\ar[d] & G \ar[r]\ar[d] & T_E \ar[r] & 0\\
& 0 & 0 & &
}\end{split}\end{equation}

where $G:=\coker(\iota(s))$ and $L=\det(E)$; we argue that $G$ is torsion free, so that $G\simeq\cali_{C'}\otimes L$ for some l.c.m. curve $C'\subset C$. 

Indeed, assume that $G$ is not torsion free and assume that $P\hookrightarrow G$ is the maximal torsion subsheaf, so that $G/P$ is torsion free; the exact sequence in the middle column implies that $\inext^p(G,\ox)=0$ for $p>1$, thus $\inext^p(P,\ox)=0$ for $p>1$ as well (since $\inext^3(G/P,\ox)=0$); it follows that $P$ and $\inext^1(P,\ox)$ must be a 2-dimensional sheaf. On the other hand, since $\dim T_E\le1$, one can dualize the exact sequence in the bottom line and conclude that $G^{\vee}\simeq L^{\vee}$ and
\begin{equation} \label{exts-serre}
0\longrightarrow \inext^1(G,\ox) \longrightarrow \inext^1(\cali_C\otimes L,\ox) \longrightarrow \inext^2(T_E,\ox) \longrightarrow 0,
\end{equation}
since $\inext^p(G,\ox)=0$ for $p>1$. This means that $\dim\inext^1(G,\ox)\le1$, providing a contradiction. Therefore, $G$ does not admit a torsion subsheaf.

In general, the quotient sheaf $T_E$ is not pure dimensional. Our next result characterizes those torsion free sheaves $E$ for which $T_E$ has pure dimension 1.

\begin{lemma} \label{lem:pure1d}
Assume that the pairs $(E,s)$ and $(C,\xi)$ correspond via the Serre correspondence outlined in Theorem \ref{thm:serre}. The scheme $C$ is l.c.m. if and only if $T_E$ is a pure 1-dimensional sheaf.
\end{lemma}
\begin{proof}
Dualizing the bottom line of the diagram in display \eqref{snake-global-sect} yields the first of the following isomorphisms
$$ \inext^3(T_E,\ox)\simeq\inext^2(\cali_C\otimes L,\ox) \simeq \inext^3(\OC,\ox); $$
the left most sheaf is 0-dimensional, so one can disregard the twist by  the line bundle $L$. Therefore, $T_E$ contains a 0-dimensional subsheaf if and only if $\OC$ also does, meaning that $C$ is not l.c.m.
\end{proof}

Let $U\hookrightarrow\OC$ be the maximal 0-dimensional subsheaf, so that $\inext^3(\OC,\ox)\simeq\inext^3(U,\ox)$. As a by-product of the previous proof, we also conclude that $\inext^3(T_E,\ox)\simeq\inext^3(U,\ox)$. In other words, the 0-dimensional components of the support of $T_E$ are always contained in the 0-dimensional components of $C$, regardless of the choice of section $s$. 


\section{Instanton sheaves on Fano threefolds} \label{sec:instantons}

Let $X$ be a Fano threefold of Picard rank one and of index $i_X$, following all the notation and definitions posed in Section \ref{sec:fano}.

The key point of the present paper is the introduction of the following definition, which generalizes Faenzi's and Kuznetsov's definitions of instanton bundles on a Fano threefold, compare with \cite[Definition 1]{F} and \cite[Definition 1.1]{K}.

\begin{definition}\label{defn:instanton}
An \emph{instanton sheaf} $E$ on $X$ is a $\mu$-semistable sheaf with first Chern class $c_1(E)=-r_X$ and such that:
\begin{equation}\label{instantonic-condition}
H^1(E(-q_X))=H^2(E(-q_X))=0.
\end{equation}
The \emph{charge} of $E$ is defined to be $c_2(E)$.
\end{definition}

The $\mu$-semistability condition rules out $\ox(-r_X)\oplus\ox^{\oplus(r-1)}$ as instanton sheaf when $r_X=1$ (ie. when $i_X$ is odd); however, $\ox^{\oplus r}$ is considered an instanton sheaf when $r_X=0$ (ie. when $i_X$ is even).

\begin{remark}
When $X=\p3$ this definition is, in general, a bit more restrictive that the definition of instanton sheaves adopted in \cite{J-inst,GJ,JMaiT,JMT2}; in these references, an instanton sheaf on $\p3$ was defined as a torsion free sheaf $E$ with $c_1(E)=0$ and 
$$ h^0(E(-1)) = h^1(E(-2)) = h^2(E(-2)) = h^3(E(-3)) = 0. $$
Using this definition, one can find examples of instanton sheaves of rank 4 and larger that are not $\mu$-semistable, see \cite[Example 3]{J-inst}.
However, both definitions are equivalent when $\rk(E)=2$, since every rank 2 sheaf on $\p3$ satisfying the conditions above is automatically $\mu$-semistable.  
\end{remark}

The following technical results will be useful later on.

\begin{lemma}\label{instantonic}
Let $E$ be an instanton sheaf of rank $r$.
\begin{enumerate}
\item $H^0(E(-n))=0 \ \forall\: n\ge 1-r_X$ and $H^3(E(n))=0, \ \forall\: n\ge -i_X+1$.
\item $H^i(E(-q_X))=0, \ \Ext^i(E,\calo_X(-q_X-r_X))=0, \ \forall\: i.$
\end{enumerate}
\end{lemma}

In particular, we conclude that $\chi(E(-q_X))=0$.

\begin{proof}
Since $E$ is $\mu$-semistable, $\Hom(\calo_X(n),E)=0$ whenever $n>\frac{-r_X}{r}$ and since $r_X=0$ or $1$, this happens if and only if $n\ge 1-r_X$. An equivalent argument leads to $\Hom(E,\calo_X(n))=0$ whenever $n<\frac{-r_X}{r}$ that is to say, whenever $n\ge-i_X+1$.
Therefore $H^0(E(n))=0$ for $n\ge 1-r_X$ and, by duality, $\Ext^3(\calo_X(-n),E)\simeq H^3(E(n))=0 $ for $n\ge-i_X+1$. Since $1-r_X\le q_X\le i_X-1$ we get \mbox{$H^i(E(-q_X))=0$,} for $i=0,3$ and this together with (\ref{instantonic-condition}) leads to $H^i(E(-q_X))=0, \ \forall i$.
By Serre duality, we have isomorphisms 
$$H^i(E(-q_X))\simeq \mathrm{Ext}^{3-i}(E(-q_X), \omega_X))^*\simeq \mathrm{Ext}^{3-i}(E, \ox(-q_X-r_X))^*,$$ therefore $\Ext^i(E,\ox(-q_X-r_X))=0, \ \forall\: i$.
\end{proof}

From these computations we determine the value of the third Chern class of instanton sheaves.

\begin{corollary}\label{c3-inst}
Let $E$ be an instanton sheaf of rank $r$. 
\begin{enumerate}
    \item If $i_X>1, \ \chi(E(-q_X))=\frac{c_3(E)}{2}=0,$ hence $c_3(E)=0$.
    \item If $i_X=1,\ \chi(E(-q_X))=\chi(E)=(r-2)+\frac{c_3(E)}{2}=0$ hence $c_3(E)=2(2-r)$.
\end{enumerate}
\end{corollary}

\begin{proof}
$\chi(E(-q_X))=0$ by Lemma \ref{instantonic}; by Grothendieck--Riemann--Roch theorem we compute $\chi(E(-q_X))=\frac{c_3(E)}{2}$ whenever $i_X>1$ so that $c_3(E)=0$ and \mbox{$\chi(E)=(r-2)+\frac{c_3(E)}{2}$} for $i_X=1$ so that $c_3(E)=2(2-r)$. 
\end{proof}

The main motivation behind Definition \ref{defn:instanton} is that non locally free instanton sheaves naturally arise as degenerations of locally free ones. When $X=\p3$, this phenomenon has been studied in \cite{JMT1,JMT2}. To see it in greater generality, let us consider some explicit examples of rank 2 instanton sheaves. 

Let $C:=L_1\sqcup\cdots\sqcup L_n$ be the disjoint union of lines in $X$, and set $G:=\cali_C(q_X-1)$; note that for $p=1,2$, we have
$$ H^p(G(-q_X)) = H^p(\cali_C(-1)) \simeq \bigoplus_{k=1}^n H^{p-1}(\calo_{L_k}(-1)) = 0. $$
We can then consider extensions of the form
\begin{equation} \label{thooft}
0 \longrightarrow \ox(-q_X-r_X+1) \longrightarrow E \longrightarrow \cali_C(q_X-1) \longrightarrow 0;
\end{equation}
clearly, $c_1(E)=-r_X$ and one easily checks that $H^p(E(-q_X))=0$ for $p=1,2$. Applying Lemma \ref{lem:stable rk 2}, we verify that $E$ is $\mu$-semistable when $i_X\ge2$; however, such sheaves are always properly $\mu$-semistable when $i_X=2$ and are not $\mu$-semistable when $i_X=1$. Therefore, E is a rank 2 instanton sheaf provided $i_X\ge2$. Inspired by the traditional nomenclature for $X=\p3$, instanton sheaves given by an extension as in display \eqref{thooft} are called \emph{'t Hooft instantons}; the charge of a 't Hooft instanton sheaf corresponding to $n$ lines is $n-1$

\begin{example}\label{ex:family}
Here is an example of a family of rank 2 locally free instanton sheaves degenerating into a non locally free one. Assume that $i_X\ge2$, and let $C$ be a disjoint union of $n\ge2$ lines in $X$, as above.

Since
$$ \Ext^1(\cali_C(q_X-1),\ox(-q_X-r_X+1)) = \Ext^1(\cali_C,\omega_X(2)) \simeq H^0(\omega_C(2)) =  \bigoplus_{k=1}^n H^{0}(\calo_{L_k}), $$
we can consider extension classes $\xi_t=(1,\dots,1,t)$ with $t\in\C$, inducing a family of instanton sheaves $E_t$, parametrized by $\C$.

Note that $E_t$ is locally free when $t\ne0$, since $\xi_t$ is non vanishing in this case. On the other hand $\xi_0$ vanishes along $L_n$, so the corresponding 't Hooft instanton sheaf $E_0$ is not locally free. Note that $E_0$ satisfies the following short exact sequence
$$ 0 \longrightarrow E_0 \longrightarrow F \longrightarrow \calo_L(q_X-1) \longrightarrow 0, $$
where $F$ is a locally free 't Hooft instanton sheaf of charge $n-1$.
\end{example}

When $i_X$ is even, instanton sheaves of rank larger than 2 can easily be produced using rank 2 instantons and ideal sheaves, via the following claim.

\begin{lemma}
Assume that $i_X=2,4$, so that $r_X=0$. 
If $E_1$ and $E_2$ are instanton sheaves, then any extension of $E_1$ by $E_2$ is also an instanton sheaf.
\end{lemma}
\begin{proof}
If $E$ is an extension of $E_1$ by $E_2$, then it is easy to check that $E$ satisfies the cohomological conditions in Definition \ref{defn:instanton}. Since $E_1$ and $E_2$ are $\mu$-semistable sheaves with vanishing slope, then so is $E$. 
\end{proof}

Next, we consider the generalization of a definition first introduced in \cite[Definition 6.1]{HLa} for $X=\p3$, and further studied in \cite{GJ,JMaiT}.

\begin{definition}\label{defn:0-instanton}
A \emph{rank 0 instanton sheaf} on a Fano threefold $X$ is a pure 1-dimensional sheaf $T$ satisfying $h^0(T(-q_X))=h^1(T(-q_X))=0$.
\end{definition}

If $T$ is a rank 0 instanton sheaf on $X$, then $\chi(T(t))=d\cdot(t+q_X)$, and the coefficient $d$ is called the \emph{degree} of $T$. Moreover, Lemma \ref{lem:T-sst} implies that $T$ is always semistable.

\begin{proposition} \label{dual rk 0}
If $T$ is a rank 0 instanton sheaf, then so is $T^{\rm D}\otimes\omega_X^{-1}(-r_X)$, where $T^{\rm D}:=\inext^2(T,\omega_X)$.
\end{proposition}

\begin{proof}
Note first that $T^{\rm D}$ is pure 1-dimensional sheaf. Using the spectral sequence \eqref{ss-loc-glob-ext} for local to global Ext's, we check that, for $p=0,1$
$$ H^p(T^{\rm D}\otimes\omega_X^{-1}(-r_X-q_X)) = H^p(\inext^2(T,\ox(-q_X-r_X))) \simeq \Ext^{p+2}(T,\ox(-q_X-r_X)). $$
Serre duality yields the isomorphism
$$ \Ext^{p+2}(T,\ox(-q_X-r_X)) \simeq H^{1-p}(T(-q_X))^* $$
and the latter vanishes by the instantonic condition on $T$.
\end{proof}

\begin{lemma}\label{lem:coho-t}
If $T$ is a rank 0 instanton sheaf, then $h^0(T(-q_X-n))=0$ and $h^1(T(-q_X+n))=0$ for every $n\ge0$.
\end{lemma}

\begin{proof}
Given a rank 0 instanton sheaf $T$, let $S\subset X$ be a hyperplane section transversal to the support of $T$ (ie. $\dim(\supp(T)\cap S)=0$), so that $\intor^1(T,\calo_S)=0$. This implies that we can twist the exact sequence $0\to\ox(-1)\to\ox\to\calo_S\to0$ by $T(k)$ to obtain the short exact sequence
$$ 0 \longrightarrow T(k-1) \longrightarrow T(k) \longrightarrow T\otimes\calo_S(k) \longrightarrow 0. $$
Taking cohomology, we conclude that $h^0(T(k-1))=0$ whenever $h^0(T(k))=0$, while $h^1(T(k))=0$ whenever $h^1(T(k-1))=0$, since $\dim(T\otimes\calo_S)=0$. The desired claims follow by induction. 
\end{proof}

We are now interested in detecting when a l.c.m. curve $C$ is such that the structure sheaf $\calo_C$ is a rank 0 instanton. We refer to a curve $C$ of such a kind as an \textit{instanton curve}.

\begin{lemma}\label{lem:instanton curves}
Let $X$ be a Fano threefold of Picard rank one.
\begin{enumerate}
\item[(1)] There are no instanton curves when $i_X=1,4$.
\item[(2)] When $i_X=2,3$, every instanton curve $C$ of degree $d$  fits in a short exact sequence of the form
\begin{equation}\label{reduction}
0\longrightarrow \calo_l \longrightarrow \calo_C \longrightarrow \calo_{C'} \longrightarrow 0
\end{equation}
where $l\subset X$ is a line and $C'$ is an instanton curve of degree $d-1$.
\end{enumerate}
\end{lemma}

\begin{proof}
The fact that there are no instanton curves on a Fano threefold $X$ of index $i_X=1$ is simply due to the fact that we can not have a projective algebraic curve $C\subset X$ such that $H^0(\calo_C)=0$ (the restriction map $H^0(\calo_X)\to H^0(\calo_C)$ is necessarily an injection).

Let us now treat the cases $q_X>0$. Consider an instanton curve $C\subset X$ of degree $d$; if $C$ is reduced, then 
$$ h^0(\calo_C) = \chi(\calo_C) + h^1(\calo_C) \ge d\cdot q_X ,$$
thus $C$ has at least $d\cdot q_X$ connected components.

If $q_X=2$ (i.e. $X=\p3$), then this is impossible since $C$ can have at most $d$ connected components.

If $q_X=1$, then $C$ is a reduced curve of degree $d$ with with at least $d$ connected components; the only possibility is for $C$ to be a disjoint union of lines which ensures that $\calo_C$ fits in a short exact sequence of the form (\ref{reduction}).

If $C$ is not reduced, let $C_{\rm red}$ be its reduction. This latter is a l.c.m. curve satisfying $h^1(\calo_{C_{\rm red}}(-q_X))=0$ and $h^0(\calo_{C_{\rm red}}(-q_X))=0$ (since $q_X>0$ and $C_{\rm red}$ is reduced and l.c.m.): in other words $C_{\rm red}$ is an instanton curve.

In particular, there are no instanton curves $C$ for $i_X=1,4$
and for $i_X=2,3,$ $C_{\rm red}$ is a disjoint union of $d'<d$ lines. It follows that 
$$ \calo_C\simeq \bigoplus_{j=1}^{d'} \calo_{C_j} $$
where $\ell_j:=(C_j)_{\rm red}$ are disjoint lines, and $H^i(\calo_{C_j}(-1)), \ i=0,1$. To prove our claim it is therefore enough to show that each $\calo_{C_j}$ fits in a short exact sequence of the form (\ref{reduction}); this is obtained applying Lemma \ref{instantonic-multiple-lines} below.
\end{proof}

\begin{lemma}\label{instantonic-multiple-lines}
Let $C$ be a multiple structure of degree $d$ on a line $\ell\subset X$. Suppose that $H^i(\calo_C(-1))=0$ for $i=0,1$. Then $C$ fits in a short exact sequence of the form 
$$ 0\rightarrow \calo_{\ell}\rightarrow \calo_C\rightarrow \calo_{C'}\rightarrow 0$$ 
where $C'$ is a multiple structure of degree $d-1$ on ${\ell}$ such that $H^i(\calo_{C'}(-1))=0$ for \mbox{$i=0,1$}.
\end{lemma}

\begin{proof}
According to \cite{Banica-Forster}, a curve $C$ satisfying the hypotheses above admits a filtration:
\begin{equation}\label{CM-filtr}
l=C_1\subset C_2\subset \ldots \subset C_m=C
\end{equation}
where each $C_j$ is a multiple structure on $\ell$ whose sheaf of ideals $\cali_{C_{j}}$ is the kernel of a surjection $\cali_{C_{j-1}}\twoheadrightarrow L_{j-1}$ with $L_{j-1}$ a vector bundle on $\ell$. 
Furthermore there exist induced generically surjective morphisms $L_i\otimes L_j\to L_{i+j}$ for each $i,j\in \{1,\ldots m\}$ and, in particular, generically surjective maps $L_1^{\otimes j}\to L_j$. Note that since in our case $C_1=C_{\rm red}$ is the line $\ell$, each vector bundle $L_j$ splits as $L_j=\bigoplus_{i=1}^{k_j}\calo_{\ell}(a_j^i)$, \mbox{ $k_j\in \{1,2\}$} (as $l$ has codimension 2). Therefore, $C=C_m$ fits into a sort exact sequence of the form
$$ 0\rightarrow L_{m-1}\rightarrow \calo_C\rightarrow \calo_{C_{m-1}}\rightarrow 0;$$
in order to prove the lemma it is therefore sufficient to show that each summand of $L_{m-1}$ has degree 0.

Consider the second step $C_2$ of the filtration (\ref{CM-filtr}). This must satisfy $H^1(\calo_{C_2}(-1))=0$ so that $a_1^i\ge 0$ for $1\le i\le k_1$. Since we have a generically surjective morphism $L_1^{\otimes m-1}\to L_{m-1}$ we deduce that $a_{m-1}^i\ge 0, \ 1\le i\le k_{m-1}$; but as $H^0(\calo_C(-1))=0, \ H^0(L_{m-1}(-1))=0$ we have $a_{m-1}^i\le 0, \ 1\le i\le k_{m-1}$. 
The only possibility is thus $a_{m-1}^i=0 $ for $1\le i\le k_{m-1}$.
\end{proof}

From now on the l.c.m. curve of degree $d$ constructed ``inductively" via the short exact sequences of the form (\ref{reduction}), will be referred to as \textit{degree $d$ line arrangements.}
\begin{remark}\label{rmk:lines-arrangement}
From the proof of Lemma \ref{lem:instanton curves} we learn that the degree $d$ lines arrangements are the only degree $d$ l.c.m. curves $C$ such that $\chi(\calo_C)=d$.
In particular we notice that for a degree $d$ l.c.m. curve $C$, since $\chi(\calo_C(-1))=-h^1(\calo_C(-1))=- d+\chi(\calo_C)\le 0$, we always have $\chi(\calo_C)\le d$ and equality holds if and only if $C$ is a degree $d$ line arrangement. 
\end{remark}

One of the main reason that justifies our interest in rank 0 instanton sheaves is that these sheaves play a primary role in the study of non-reflexive instanton sheaves.  
It is indeed possible to construct instanton sheaves out of non-reflexive ones performing an \emph{elementary transformation} along a rank 0 instanton sheaf.

We recall that the elementary transformation consists of the following procedure.
Let $F$ be a reflexive instanton sheaf, let $T$ be a rank 0 instanton sheaf, and consider an epimorphism $q:F\onto T$. It is easy to check that $E:=\ker q$ is also an instanton sheaf. Indeed, consider the exact sequence
\begin{equation}\label{et}
0 \longrightarrow E \longrightarrow F \stackrel{q}{\longrightarrow} T \longrightarrow 0; 
\end{equation}
$E$ is $\mu$-semistable because $F$ is $\mu$-semistable and $\mu(F)=\mu(E)$; the exact sequence in cohomology (here, $p=1,2$)
$$ H^{p-1}(T(-q_X)) \longrightarrow H^p(F(-q_X)) \longrightarrow H^p(E(-q_X)) $$
implies that $h^p(E(-q_X))=0$ ($p=1,2$) since $T$ and $F$ are instanton sheaves.

It is almost immediate to prove that whenever $E$ is the elementary transformation of $F$ along $T$ then the following holds
\begin{lemma}
Let $E$ be an instanton sheaf obtained by elementary transformation of a reflexive instanton $F$ along a rank 0 instanton $T$. Then $F\simeq E^{\vee\vee}$.
\end{lemma}

\begin{proof}
Applying the functor $\inhom(\:\cdot\:, \calo_X)$ to (\ref{et}), we get $E^{\vee}\simeq F^{\vee}$ (since $T$ is one-dimensional) hence $E^{\vee\vee}\simeq F^{\vee\vee}\simeq F$.
\end{proof}
As it turns out not all the non-reflexive instanton sheaves are necessarily obtained in this way. 
We can indeed prove that for a non reflexive instanton sheaf $E$, $E^{\vee\vee}/E$ is always purely one-dimensional but not necessarily a rank 0 instanton shef; accordingly $E^{\vee\vee}$ is not necessarily an instanton sheaf either.

\begin{proposition}\label{double-dual}
Let $E$ be a non reflexive instanton sheaf of rank $r>0$. Then the following hold:
\begin{enumerate}
    \item $T_E:=E^{\vee\vee}/E$ has pure dimension one;
    \item $E$ has homological dimension 1;
    \item $H^p(E^{\vee\vee}(-q_X))=0, \ p=0,2,3$
    \item $E^{\vee\vee}$ is an instanton if and only if $T$ is a rank 0 instanton sheaf. This condition is equivalent to $c_3(E^{\vee\vee})=0$ for $i_X>1$ and $c_3(E^{\vee\vee})=2(2-r)$ for $i_X=1$.
\end{enumerate}
\end{proposition}

\begin{proof}
Since $E$ is torsion free, it injects in its double dual, leading to a short exact sequence:
\begin{equation}\label{ses-dd}
    0\rightarrow E\rightarrow E^{\vee\vee}\rightarrow T_E\rightarrow 0
\end{equation}
where $T_{E}:=E^{\vee\vee}/E$ is a torsion sheaf supported on the locus of points where $E$ fails to be reflexive. Applying the functor $\inhom(\:\cdot\:,\calo_X)$ to (\ref{ses-dd}) we obtain an exact sequence
$0\rightarrow E^{\vee}\rightarrow E^{\vee\vee}\rightarrow \inext^1(T_E,\calo_X)\rightarrow \inext^1(E^{\vee\vee},\calo_X)$
from which we deduce that $\dim(T_E)\le 1$. Indeed, since $E^{\vee}$ is reflexive, $E^{\vee}\simeq E^{\vee\vee}$ which implies that $\inext^1(T_E,\calo_X)$ injects in $\inext^1(E^{\vee\vee}\calo_X)$; but this can clearly not happen if ever $T_E$ had dimension $>1$. If ever this was the case, $\inext^{1}(T_E,\calo_X)$ would have dimension bigger then one as well leading to a contradiction since $\inext^1(E^{\vee\vee},\calo_X)$ is zero-dimensional due to the reflexivity of $E^{\vee\vee}$. 
Twisting now ($\ref{ses-dd}$) by $\calo_X(-q_X)$ and taking cohomology we get an exact sequence:
$$H^0(E^{\vee\vee}(-q_X))\rightarrow H^0(T_E(-q_X))\rightarrow H^1(E(-q_X)).$$
The l.s.t. vanishes since $E^{\vee\vee}$ is $\mu$-semistable, the r.s.t. vanishes since $E$ is an instanton; accordingly $H^0(T_E(-q_X))=0$ thus $H^0(T_E(-n))=0, \ \forall n\ge q_X$ (apply \ref{lem:coho-t}) which allows us to conclude that $T_E$ has pure dimension one ending the proof of $(1)$.
As a consequence of $(1)$ we get that $\inext^2(E,\calo_X)\simeq \inext^3(T_E,\calo_X)=0$ and since moreover $0=\inext^3(E^{\vee\vee},\calo_X)$ surjects onto $\inext^3(E,\calo_X)$, we can conclude that $E$ has homological dimension 1, proving (2). 
The long exact sequence in cohomology from ($\ref{ses-dd}$) twisted by $\calo_X(-q_X)$ now leads to
$$H^i(E(-q_X))\simeq H^i(E^{\vee\vee}(-q_X))=0, \ i=2,3,\hspace{2mm} H^1(E^{\vee\vee}(-q_X))=H^1(T_E(-q_X)).$$ These equalities lead to $(3)$ (as we have already pointed out that $H^0(E^{\vee\vee}(-q_X))$ vanishes by $\mu$-semistability) and ensure that $E^{\vee\vee}$ is an instanton if and only if $T_E$ is a rank 0 instanton. 
Finally, we compute that $\chi(T_E(-q_X))=\frac{c_3(E^{\vee\vee})}{2}$ for $i_X>1$ (resp. $\chi(T_E)=\frac{c_3(E^{\vee\vee})}{2}+r-2$ for $i_X=1$) which holds if and only if $c_3(E^{\vee\vee})=0$ (resp. if and only if $c_3(E^{\vee\vee})=2(2-r))$.
\end{proof}


\subsection{Classification of rank 1 instanton sheaves}\label{sec:rk1}

Since $\Pic(X)=\Z$, a locally free (or equivalently reflexive) instanton of rank one is uniquely determined by its first Chern class, thus the only instanton line bundle is $\calo_X(-r_X)$. Let us now consider the non locally free case.

\begin{lemma}\label{et-rank1}
Let $L$ be a non locally free instanton sheaf of rank one. Then \mbox{$L^{\vee\vee}\simeq \calo_X(-r_X)$} and $L^{\vee\vee}/L$ is a rank 0 instanton.
\end{lemma}

\begin{proof}
Applying proposition (\ref{double-dual}) (recall that in rank one reflexivity is equivalent to local freeness) we have that $L$ always fits in a short exact sequence of the form:

\begin{equation}\label{ses-tf}
    0\rightarrow L\rightarrow L^{\vee\vee}\rightarrow T\rightarrow 0
\end{equation}
with $T$ being a torsion sheaf of pure dimension one.
Accordingly, $L^{\vee\vee}$ is a rank one reflexive sheaf with $c_1(L^{\vee\vee})=c_1(L)=-r_X$, that is to say $L^{\vee\vee}\simeq \calo_X(-r_X)$. Since $\calo_X(-r_X)$ is an instanton, one can easily check that is a rank 0 instanton sheaf.
\end{proof}

We therefore understand that the classification of rank 1 instanton sheaves reduces to the classification of rank 0 instanton sheaves $T$ admitting an epimorphism $\calo_X(-r_X)\twoheadrightarrow T$.
This latter can be attained as a consequence of Lemma \ref{lem:instanton curves}.
\begin{proposition}\label{rk1}
Let $L$ be a rank one instanton sheaf of charge $d$ on a Fano threefold $X$ with Picard rank one. The following hold:
\begin{enumerate}
\item if $i_X=3,4$, then $d=0$ and $L\simeq \calo_X(-r_X)$;
\item if $i_X=1,2$, we have $L\simeq\calo_X(-r_X)$ whenever $d=0$ whilst for $d>0$ $L$ always fits in a short exact sequence of the form:
$$ 0\rightarrow L\rightarrow L'\rightarrow \calo_{\ell}\rightarrow 0 $$
for a line ${\ell}\subset X$ and $L'$ a rank one instanton sheaf of charge $d-1$.
\end{enumerate}
\end{proposition}

\begin{proof}
By Lemma \ref{et-rank1}, the classification of rank one instanton sheaves $L$ of charge $d$, reduces to the classification of degree $d$ rank 0 instanton sheaves $T$ admitting a surjection $\calo_X(-r_X)\twoheadrightarrow T$, that is to say, to the classification of degree $d$ l.c.m. curves $C\subset X$ such that $\calo_C(-r_X)$ is a rank 0 instanton. But this means that  $H^i(\calo_C(-2))=0, \ i=0,1$ for $i_X=3,4$ and $H^0(\calo_C(-1))$ for $i_X=1,2$. The arguments used to prove Lemma \ref{lem:instanton curves} (i) show then that there are no rank 1 non locally free instantons of rank one on Fano threefolds of index $3$ or $4$.
Similarly, from Lemma \ref{lem:instanton curves} (2) and Remark \ref{rmk:lines-arrangement}, we know that the only degree $d$ l.c.m. curves such that $H^i(\calo_C(-1))=0$ are the degree $d$ lines arrangements. This proves point (2).
\end{proof}

\begin{remark}
Proposition \ref{rk1} can be rephrased by saying that each rank one instanton $L$ of charge $d>0$, on a Fano threefold $X$ of index $i_X\in \{1,2\}$, is always isomorphic to $\cali_C(-r_X)$ for $C$ a degree $d$ line arrangement. Recall that a curve $C$ of such a kind is supported of $d'\le d$ disjoint lines and can be constructed ``inductively" from an extension:
$$ 0\rightarrow \calo_{\ell}\rightarrow \calo_C\rightarrow \calo_C'\rightarrow 0$$
with $C'$ a degree $d-1$ line arrangement. As a consequence, for $i_X=2$, rank one instantons of strictly positive charge coincide with ideal sheaves of instanton curves.
\end{remark}


\subsection{Classification of rank 2 instanton sheaves}\label{sec:rk2}

In general, the double dual $E^{\vee\vee}$ of a torsion free sheaf $E$ is a  reflexive (possibly non locally free) sheaf; if $E$ is an instanton sheaf, we have that $c_1(E^{\vee\vee})=-r_X$ and that $E$ is $\mu$-semistable, but $E^{\vee\vee}$ may not satisfy the instantonic vanishing conditions.

The main result of this section guarantees that $E^{\vee\vee}$ is a locally free instanton sheaf when $\rk(E)=2$. Recall that, since $X$ has Picard rank one, we have in this case an isomorphism:
\begin{equation}\label{iso-refl}
E^{\vee\vee}\simeq E^{\vee}(-r_X).
\end{equation}
In particular this implies that the Serre's duality establishes isomorphisms:
\begin{equation}\label{Serre-duality-refl}
    H^i(E^{\vee\vee}(n))\simeq H^{3-i}(E^{\vee}(-n-i_X))^*\simeq H^{3-i}(E^{\vee\vee}(-2q_X-n))^*, \ i=0,3. 
\end{equation}

When $X$ is a Fano threefold, Lemma \ref{lem:stable rk 2} says that a normalized rank 2 torsion free sheaf $E$ on X is $\mu$-semistable if and only if $h^3(E(-i_X+1))=0$. In addition, note that $\mu$-semistability implies that $h^0(E(r_X-1))=0$, when we assume that $c_1(E)=r_X$.

Let us now focus on rank 2 instantons.
To begin with, we show that, in the rank 2 case, the reflexivity of an instanton implies its local freeness. 
\begin{lemma}\label{rank2-bundle}
Let $E$ be a rank 2 reflexive instanton. Then $E$ is locally free.
\end{lemma}
\begin{proof}
Applying Lemma \ref{c3-inst}, $\chi(E(-q_X))=0$ leading to $c_3(E)=0$.
The proof of \cite[Proposition 2.6]{H-reflexive} applies verbatim to arbitrary Fano threefolds of Picard rank one which allows to conclude that $c_3(E)=0$ if and only if $E$ is locally free. 
\end{proof}

We are finally ready to prove our classification of non locally free rank 2 instanton sheaves.

\begin{theorem}\label{double-dual-bundle}
Let $E$ be a rank 2 instanton sheaf. Then $E^{\vee\vee}$ is an instanton bundle and $T_E:=E^{\vee\vee}/E$ is a rank 0 instanton sheaf whenever $T_E\ne 0$.
\end{theorem}

\begin{proof}
By Propositions \ref{double-dual} and \ref{rank2-bundle}
it is enough to prove that for a non locally free instanton $E$, $h^1(E^{\vee\vee}(-q_X))=0$.
Consider the local to global spectral sequence
$$E^{p,q}_2=H^p(\inext^q(E^{\vee\vee},\calo_X(-q_X-r_X)))\Rightarrow \Ext^{p+q}(E^{\vee\vee},\calo_X(-q_X-r_X)).$$
$\inhom(E^{\vee\vee},\calo_X(-q_X-r_X))\simeq E^{\vee\vee}(-q_X)$ hence $E^{p,0}=0$ for $p\ne 1$ which implies that the spectral sequence already degenerates at the $r=2$ sheet.
Therefore $$\Ext^1(E^{\vee\vee},\calo_X(-q_X-r_X))\simeq H^0(\inext^1(E^{\vee\vee},\calo_X(-q_X-r_X)))\oplus H^1(E^{\vee\vee}(-q_X));$$
since $\Ext^1(E^{\vee\vee},\calo_X(-r_X-q_X))\simeq H^2(E^{\vee\vee}(-q_X))^*=0$, we then deduce that  $H^0(\inext^1(E^{\vee\vee},\calo_X(-q_X-r_X)))=0$ and $H^1((E^{\vee\vee}(-q_X))=0.$ 
This ensure that $E^{\vee\vee}$ is an instanton bundle.
\end{proof}

\begin{remark}
Theorem \ref{double-dual-bundle} generalizes \cite[Proposition 3,1, Theorem 3.5]{Druel} and \cite[Theorem 1,2]{Qin1}, \cite[Theorem 1.3]{Qin2}. 
\end{remark}

\begin{remark}
It is worth pointing out that if in the definition of instanton, we replace $\mu$-semistability with a weaker cohomological condition, Theorem \ref{double-dual-bundle} no longer holds.
We can prove this with the following counterexample. 
Let us consider a Fano threefold $X$ of index $i_X=1$ and a sheaf $E\in Coh(X)$ defined by:
$$ 0\rightarrow E\rightarrow \calo_X\oplus \calo_X(-1)\rightarrow \calo_p\rightarrow 0$$
with $p\in X$ a point.
$E$ is a $\mu$-unstable rank 2 torsion free (it is destabilized by $\cali_p$) such that $H^i(E)=0$ for $0\le i\le 2$, but $E^{\vee\vee}/E\simeq \calo_p$.
\end{remark}

\begin{corollary}
Let $E$ be a rank 2 non-locally free instanton. Then the sheaf \mbox{$S_E:=\inext^1(E,\ox(-r_X))$} is a rank 0 instanton. 
\end{corollary}

\begin{proof}
Whenever $E$ is not locally free, it fits in an exact sequence of the form \eqref{ses-dd} and since $E^{\vee\vee}$ is locally free, we get an isomorphism:
$$ \inext^1(E,\ox(-r_X)) \simeq \inext^2(T_E,\ox(-r_X)).$$
Since $T_E$ is a rank 0 instanton, by Lemma \ref{dual rk 0}, 
$T_E^{\rm D}\otimes\omega_X^{-1}(-r_X)$ is a rank 0 instanton as well. The statement of the corollary then follows from the fact that we have an isomorphism $\inext^2(T_E,\ox(-r_X))\simeq T_E^{\rm D}\otimes\omega_X^{-1}(-r_X)$.
\end{proof}

\begin{remark}
If $E$ is a rank 2 non locally free instanton, we then have an equality \mbox{$\supp(T_E)=\sing(E)$} for $\ T_E:=E^{\vee\vee}/E$.
Note that this might not hold for arbitrary rank since a priori $\supp(T_E)\subset \sing(E)$ and equality holds if and only if $E^{\vee\vee}$ is locally free.
\end{remark}

Summing up these last results, we can affirm that in rank 2 an instanton sheaf $E$ is either locally free or $\sing(E)$ has pure dimension one and $E$ is obtained by elementary transformation of an instanton bundle along a rank 0 instanton supported on $\sing(E)$. 
Elementary transformation of rank 2 instanton bundles has been widely used in \cite{JMT2} to construct and study families of non-locally free instantons on $X=\p3$. Most importantly for the present paper, Faenzi made a very interesting use of this construction in \cite{F}: in loc. cit. elementary trasnformation is indeed used to prove the existence of rank 2 instanton bundle of charge $k, \forall\: k\ge 2$ on Fano threefolds of index 2. Mimicking this approach, we can state the following theorem:

\begin{theorem}\label{deformation-et}
Let $X$ be a Fano threefold of index $i_X$. Assume that $F$ is a rank 2 locally free instanton sheaf, and $\ell\subset X$ is a line such that the following hypotheses hold:
\begin{itemize}
    \item $F$ is unobstructed, i.e. $\Ext^2(F,F)=0$;
    \item $\caln_{\ell/X}\simeq\calo_{\ell}(q_X-1)\oplus \calo_{\ell}(i_X-q_X-1)$ and $F|_{\ell}\simeq \calo_{\ell}\oplus \calo_{\ell}(-r_X)$.
\end{itemize}
Then $X$ admits rank 2 locally free instanton sheaves of charge $k$ for every $k\ge c_2(F)$.
\end{theorem}

\begin{proof}
The induction argument presented in \cite[Theoreom D]{F} applies to any Fano threefold $X$ carrying an instanton bundle $F$ and a line $\ell$ satisfying the hypotheses of the theorem. We summarize here the main steps of the proof. For any pair $(F,\ell)$ as above, we get the existence of an epimorphism $\phi: F\twoheadrightarrow \calo_{\ell}(q_X-1)$ whose kernel $E$ is a non locally free instanton (as $\calo_{\ell}(q_X-1)$ is a rank 0 instanton) of charge $c_2(F)+1$. One then proves that the assumptions made on $(F,\ell)$ ensure then that $\Ext^2(E,E)=0$ and that a general deformation of $E$ is an instanton bundle.
This is done showing at first that a general non-locally free deformation of $E$ is still an instanton singular along a line and obtained from a deformation of $(\phi,F,\ell)$. We then compute that the family of instantons whose singular locus is a line has dimensions strictly less then $\ext^1(E,E)$. As a consequence $E$ deforms to a locally free sheaf and by semicontinuity, a general deformation is a non-obstructed instanton bundle.
By induction we therefore get existence of unobstructed instanton bundle of rank 2 for each charge $\ge c_2(F)$.
\end{proof}
The theorem always applies on Fano threefold of index $i_X\ge 2$.
\begin{itemize}
    \item On $X\simeq \D{P}^3$ a t'Hooft instanton of charge $1$ and a general line lead to the existence of rank 2 instantons of charge $k$ for each $k\ge 1$.
    The same result is proved, with different techniques in \cite{JMT2}.
    \item For $i_X=3$ the spinor bundle and any line $\ell\subset X$ ensure the existence of instanton bundles of charge $k\ge 1$.
    \item The case $i_X=2$ was treated in \cite{F} where the existence of instanton bundles of charge $k\ge 2$ is proved. 
    \end{itemize}
Fano threefolds of index one for which the theorem applies are treated in \cite[Theorem 3.7]{BF}.

We end the section characterizing the dimensions of the intermediate cohomology groups $H^i(E(n)), \ i=1,2$ of rank 2 instantons (the groups $H^i(E(n))$ for $i=0,3$ where studied in Lemma \ref{instantonic}).
\begin{lemma}\label{cohom-interm}
Let $E$ be a rank 2 instanton and let $m_0$ denote the smallest integer such that $m_0H$ is very ample. Then the following hold:
\begin{enumerate}
    \item $H^1(E(-q_X-n))=0$ for $n=m_0$ and $\forall n\ge 2m_0$
    \item $H^2(E(-q_X+n))=0$ for $n=m_0$ and $\forall n\ge 2m_0$
\end{enumerate}
\end{lemma}
\begin{proof}
Let $D$ be a general element in the linear system $|m_0H|$. By generality assumption $E|_D$ is $\mu$-semistable (see \cite{Ma}), therefore  $H^0(E|_D(-q_X))=0$, and $H^2(E|_D(-q_X+m_0))\simeq Hom(E|_D,\calo_D(-q_X-r_X))^*=0$. Taking cohomology in the short exact sequence:
\begin{equation}\label{restriction}
0\rightarrow E(-q_X-m_0)\rightarrow E(-q_X)\rightarrow E|_{D}(-q_X)\rightarrow 0
\end{equation}
we therefore get that $H^1(E(-q_X-m_0))=0$; twisting then (\ref{restriction}) by $\calo_X(m_0)$ and taking cohomology, we obtain $H^2(E(-q_X+m_0))$.
These arguments apply to the letter to any general divisor $D\in |nH|, n\ge 2m_0$ since under these assumption $D$ is very  ample and $E|_D$ is $\mu$-semistable.
\end{proof}


\subsection{Stability of rank 2 instanton sheaves}

Clearly, when $i_X$ is odd, every rank 2 instanton sheaf $E$ on $X$ is $\mu$-stable, simply because $\mu$-stability coincides with $\mu$-semistability when $c_1(E)=-1$. 

In \cite{JMaiT}, the authors showed that every non trivial rank 2 instanton sheaf on $X=\p3$ is stable, so that $\op3^{\oplus2}$ is the only properly semistable (meaning semistable but not stable) rank 2 instanton sheaf. In addition, a rank 2 instanton sheaf is properly $\mu$-semistable only when $E^{\vee\vee}=\op3^{\oplus2}$.

The situation is quite different for Fano threefolds of index 2. Indeed, let $E$ be a properly $\mu$-semistable rank 2 instanton sheaf on a Fano threefold $X$ with $i_X=2$. When $E$ is locally free, this is equivalent to say that $h^0(E)>0$; choosing a non trivial section $s\in H^0(E)$, we obtain the exact sequence
\begin{equation}\label{st mu-ss}
0 \longrightarrow \ox \longrightarrow E \longrightarrow \cali_C \longrightarrow 0,    
\end{equation}
and it is easy to check that $\cali_C$ is a rank one instanton, meaning that $C$ is an instanton curve; recall that the latter have been classified in Lemma \ref{lem:instanton curves}. Summing up, we proved the following claim.

\begin{lemma}
Let $E$ be a rank 2 locally free instanton sheaf on a Fano threefold $X$ of index $i_X=2$. If $E$ is properly $\mu$-semistable, then $E$ fits into an exact sequence as in display \eqref{st mu-ss} where $C$ is an instanton curve. In particular, such sheaves are not semistable.
\end{lemma}

When $E$ is not locally free, we have that $E^{\vee\vee}$ is a properly $\mu$-semistable locally free instanton sheaf of rank 2; taking the unique (up to scalar multiple) non trivial section $s\in H^0(E^{\vee\vee})$, we get the following commutative diagram
\begin{equation}\label{diag1}
\begin{tikzcd}
&0\arrow[d] & 0 \arrow[d] & \\
  &\cali_{C_1}\arrow[d]\arrow[r] &\calo_X \arrow[d, "s"] \arrow[dr,"\phi"] &   \\
0  \arrow[r] & E\arrow[r]\arrow[d, two heads] & E^{\vee\vee} \arrow[d, two heads]\arrow[r, "q"] & T \arrow[d, "Id"] \arrow [r] &0 \\
0\arrow[r]  & \cali_{C_2}\arrow[r]\arrow[d] & \cali_{C}\arrow[r]\arrow[d] & \coker(\phi) \arrow [r]& 0\\
&0  & 0  & 
\end{tikzcd}
\end{equation}
Here $\phi=q\circ s$; clearly, the kernel of $\phi$ is the ideal sheaf of a pure 1-dimensional scheme, which we denote by $C_1$; remark that if $q=0$, then $C_1$ is empty. Since the cokernel of the inclusion $\cali_{C_1}\hookrightarrow E$ must also be torsion free, we complete the leftmost column. Note that
\begin{equation}\label{vanish-ideal}
H^1(\cali_{C_1}(-1)) = H^2(\cali_{C_2}(-1)) = 0 ~~ \Longleftrightarrow ~~ H^0(\calo_{C_1}(-1)) = H^1(\calo_{C_2}(-1)) = 0.
\end{equation}

\begin{lemma}\label{semistable-index2}
Let $E$ be a rank 2 instanton sheaf on a Fano threefold of index 2. If $E$ is properly semistable, then $E$ is S-equivalent to $\cali_{C_1}\oplus\cali_{C_2}$ where $C_1$ and $C_2$ are instanton curves of the same degree.
\end{lemma}
\begin{proof}
As we have seen above, the hypothesis imply that $E$ must be an extension of ideal sheaves $\cali_{C_1}$ and $\cali_{C_2}$ satisfying
$$ \chi(\cali_{C_j}(t))=\dfrac{1}{2}\chi(E(t))=\chi(\ox(t))-\dfrac{c_2(E)}{2}(t+1). $$
In particular, $\chi(\cali_{C_j}(-1))=0$; the vanishings in display \eqref{vanish-ideal} imply that $C_1$ and $C_2$ must be instanton curves.
\end{proof}

From Lemma \ref{semistable-index2} we deduce the following corollaries.

\begin{corollary}
There are no properly semistable rank 2 instanton of odd charge.
\end{corollary}
 
\begin{proof}
Each semistable instanton $E$ is S-equivalent to a sheaf of the form $\cali_{C_1}\oplus\cali_{C_i}$ where $C_1$ and $C_2$ are instanton curves of the same degree $d$; thus $c_2(E)=2d$.
\end{proof}

\begin{corollary}
There exists no properly semistable locally free instanton sheaves of charge $>0$.
\end{corollary}

\begin{proof}
If $E$ is semistable and $c_2>0$ then $E$ fits in 
$$0\rightarrow \cali_{C_1}\rightarrow E\rightarrow \cali_{C_2}\rightarrow 0$$
with $C_i$ being instanton curves of degree $\deg(C_1)=\deg(C_2)=\frac{c_2(E)}{2}.$
From this short exact sequence we compute that $E$ has depth 2 along all points $x\in C_1$ hence $E$ can not be locally free. 
\end{proof}

At the beginning of the section we observed that there are no properly $\mu$-semistable non locally-free instanton sheaves on Fano threefolds of odd index (since on these varieties $\mu$-semistable rank 2 bundles are $\mu$-stable) and that each properly $\mu$-semistable non locally instanton $E$ on $\D{P}^3$ satisfies $E^{\vee\vee}\simeq \calo_{\D{P}^3}^{\oplus 2}$.
The Fano threefolds $X$ of index $i_X=2$ are the only ones carrying families of properly $\mu$-semistable non locally free instanton sheaves $E$ such that $c_2(E^{\vee\vee})>0$.
Moreover, even if every properly $\mu$-semistable instanton bundle $F$ with $c_2(F)>0$ is Gieseker unstable,
the instanton sheaves $E$ obtained as elementary transformation of $F$ along rank 0 instantons might be semistable or even stable.

\begin{lemma}
Let $(F,T,q)$ be, respectively, a properly $\mu$-semistable rank 2 instanton bundle $F$ of charge $n>0$ a rank 0 instanton $T$ of multiplicity $d$ and an epimorphism $q:F\twoheadrightarrow T$. Let $E$ be the sheaf defined as $E:=\ker(q)$. Then $E$ is stable, resp. properly semi-stable, if and only if $\forall s \in H^0(F), \ s\ne 0$, $ch_2(\im(q\circ s))> \frac{n+d}{2}$, resp. if and only if $T$ is an instanton curve and $\forall s \in H^0(F),\: s\ne 0$
$\im(q\circ s)$ is an instanton curve of degree $\frac{n+d}{2}$.
\end{lemma}

\begin{proof}
As usual we start considering the short exact sequence 
$$ 0 \rightarrow E \rightarrow F \xrightarrow{q} T \rightarrow 0; $$
note that the charge of $E$ is $n+d$.

Considering a diagram analogous to the one in display \eqref{diag1}, we see that every subsheaf of $E$ with trivial determinant is the ideal sheaf $\cali_B$ of a scheme $B$ such that $\calo_B=\im(q\circ s)$ for a (hence for each) non zero section $s\in H^0(F)\simeq \D{C}$.
Denoting by $d'$ and $x$ the degree and the Euler characteristic of the curve $B$, respectively, we have:
$$ \dfrac{1}{2} P_E(t) - P_{\cali_B}(t) = -\dfrac{n+d}{2}(t+1)+d't+x =t(d'-\frac{(n+d)}{2}) +x'- \frac{(n+d)}{2} $$ 
It is therefore clear that if ever $d'> \frac{n+d}{2}$, then $E$ is stable whilst $d'< \frac{n+d}{2}$ leads to the unstability of $E$. 
In the case $d'=\frac{(n+d)}{2}$, we can never have stability since \mbox{$x'\le \frac{(n+d)}{2}$} and equality occurs if and only if $B$ is an instanton curve (or, equivalently, a degree $d'$ line arrangement). Indeed for a l.c.m. curve $B$ of degree $d'$, since $h^0(\calo_B(-1))=0$, $\chi(\calo_B(-1))=-d'+x'\le 0$ and equality holds if and only if $B$ is an instanton curve.

This shows that $E$ is stable, resp. properly semistable, if and only if for each non-zero global section $s\in H^0(F), \ ch_2(\im(q\circ s)) > \frac{n+d}{2}$ (resp.  $ch_2(\im(q\circ s)) = \frac{n+d}{2}$ and $\im(q\circ s)$ is an instanton curve).
To conclude the proof of the proposition we still need to show that if $E$ is properly semistable then $T$ itself is an instanton curve. Once again we consider the diagram \eqref{diag1} induced by $s\in H^0(F)$ (recall that $F\simeq E^{\vee\vee}$); since $\im(q\circ s)$ is an instanton curve, $\coker(q\circ s)$ is an instanton as well; moreover by the fact that $\cali_C$ surjects onto $\coker(q\circ s)$, we deduce that $\coker(q\circ s)$ must as well be isomorphic to
$\calo_{B'}$ with $B'$ an instanton curve of degree $\frac{(d-n)}{2}$. This last assertion is due to the fact that a rank 0 instanton $T'$ on a l.c.m. curve $C$ has always degree (as a $\calo_C$-module) $\deg(C)-\chi(\calo_C)\ge 0$ with equality holding if and only if $C$ is an instanton curve (indeed whenever $P_C(t)=\deg(C)t+\deg(C)$, $h^0(\calo_C(-1))=0$ implies $h^1(\calo_C(-1))=0$) and $T'\simeq \calo_C$. Therefore $T$ is an instanton curve extension of $\calo_{B'}$ by $\calo_B$. 
\end{proof}


\subsection{Instantons via Serre}

Let $X$ be a Fano threefold of Picard rank 1, index $i_X$ and take a rank 2 instanton sheaf $E$ of charge $c_2$. Following the Serre correspondence outlined in Section \ref{sec:serre}, we choose a section $s\in H^0(E(n))$ with torsion-free cokernel and obtain a short exact sequence
$$ 0 \longrightarrow \ox(-n) \longrightarrow E \longrightarrow \cali_C(n-r_X) \longrightarrow 0 $$
which yields to a l.c.m. curve $C\subset X$; its arithmetic genus $p_a(C)$ and the degree $d$ are given by
\begin{align}\label{genus-degree-inst}
p_a(C)&=1-[(d_Xn^2-d_Xnr_X+c_2)(\frac{i_X}{2}+r_X-n)+\frac{1}{2}r_X(nr_X-n^2-c_2)]\\
d&= d_Xn^2-d_Xnr_X+c_2;
\end{align}
moreover, its sheaf of ideals will satisfy the cohomological conditions:
\begin{equation}\label{instantonic-curve}
H^0(\cali_C(n-r_X-1))=0, \ H^i(\cali_{C}(n-r_X-q_X))=0, \ i=1,2.
\end{equation}

Our aim now is to characterize in detail those curves that "Serre correspond" to non locally free instanton sheaves of rank 2. 

We consider therefore a non locally free instanton sheaf $E$ and we let $n$ be a non-negative integer such that $h^0(E(n)>0$; take $s\in H^0(E(n))$ such that $\coker(s)$ is a torsion free sheaf, and let $s^{\vee\vee}\in H^0(E^{\vee\vee}(n))$ be the image of $s$ via the injective map $ H^0(E(n))\hookrightarrow H^0(E^{\vee\vee}(n)).$ According to the argument just below the diagram in display \eqref{snake-global-sect}, we obtain the following commutative diagram
\begin{equation}\begin{split}\label{diag2}
\xymatrix{
& 0\ar[d] & 0\ar[d] & & \\
& \ox(-n)\ar[d]^{s}\ar@{=}[r] & \ox(-n)\ar[d]^{s^{\vee\vee}} & & \\
0\ar[r] & E \ar[r]\ar[d] & E^{\vee\vee} \ar[r]\ar[d] & T_E \ar[r]\ar@{=}[d] & 0 \\
0\ar[r] & \cali_{C}(n-r_X) \ar[r]\ar[d] & \cali_{C'}(n-r_X) \ar[r]\ar[d] & T_E \ar[r] & 0\\
& 0 & 0 & &
}\end{split}\end{equation}
where $C$ and $C'$ are the curves corresponding to the pairs $(E,s)$ and $(E^{\vee\vee},s^{\vee\vee})$, respectively; Lemma \ref{lem:pure1d} guarantees that $C'$ is l.c.m. To figure out the associated extension classes, we note that the short exact sequence in display \eqref{exts-serre} can be rewritten as follows
\begin{equation}\label{ses-canonical}
0\rightarrow \omega_{C'}(r_X+i_X-2n)\rightarrow \omega_{C}(r_X+i_X-2n)\rightarrow \inext^2(T_E(n),\calo_X)\rightarrow 0.
\end{equation}

\begin{lemma}\label{sheaves-to-extension}
If $(C,\xi)$ corresponds to a pair $(E(n),s)$ where $E$ is a non locally free instanton sheaf of rank 2, and $n$ is a non negative integer, then the pair $(C',\xi')$ corresponding to $(E^{\vee\vee}(n),s^{\vee\vee})$ satisfies the following conditions:
\begin{enumerate}
\item $0\to\cali_C\to\cali_{C'}\to T_E(r_X-n)\to0$;
\item $\xi$ is the image of $\xi'\in H^0(\omega_{C'}(i_X+r_X-2n))$ under the inclusion:
$$ 0\rightarrow H^0(\omega_{C'}(i_X+r_X-2n))\xrightarrow{\iota} H^0(\omega_{C}(i_X+r_X-2n)) $$
\end{enumerate}
\end{lemma}
\begin{proof}
The first item is just the bottom line of diagram in display \eqref{diag2}.
As for the second item, applying $\inhom(\ \cdot\ ,\calo_X(r_X-2n))$ to the short exact sequence of point $(1)$ and taking global section, we obtain:
\begin{equation}\label{ses-canonical-gs}
 0\rightarrow H^0(\omega_{C'}(i_X+r_X-2n))\xrightarrow{\iota} H^0(\omega_{C}(i_X+r_X-2n)) \to  H^0(\inext^2(T_E(n),\calo_X))
\end{equation}
where the r.s.t and the middle term are isomorphic to \mbox{$\Ext^1(\cali_{C'}(2n-r_X),\calo_X)$} and $\Ext^1(\cali_{C}(2n-r_X),\calo_X)$, respectively. Note now that, more generally, given $\zeta'\in H^0(\omega_{C'}(i_X+r_X-2n))$ corresponding to $0\rightarrow \calo_X\xrightarrow{s'} F'\rightarrow \cali_C'(2n-r_X)\rightarrow 0$, applying the snake Lemma we obtain a commutative diagram:
\begin{equation}\begin{split}
\xymatrix{
& 0\ar[d] & 0\ar[d] & & \\
& \ox(-n)\ar[d]^{r}\ar@{=}[r] & \ox(-n)\ar[d]^{s'} & & \\
0\ar[r] & E' \ar[r]\ar[d] & F' \ar[r]\ar[d] & T_E \ar[r]\ar@{=}[d] & 0 \\
0\ar[r] & \cali_{C}(n-r_X) \ar[r]\ar[d] & \cali_{C'}(n-r_X) \ar[r]\ar[d] & T_E \ar[r] & 0\\
& 0 & 0 & &
}\end{split}\end{equation}
In particular, $F'\simeq E'^{\vee\vee},$ $s'=r^{\vee\vee}$ and $\iota(\zeta')\in H^0(\omega_{C}(i_X+r_X-2n))$ corresponds to the short exact sequence defined by the first column of the diagram. Accordingly $\xi=\iota(\xi')$ which also implies that $\xi$ lies in the kernel of the morphism
\mbox{$H^0(\omega_{C}(i_X+r_X-2n)) \to  H^0(\inext^2(T_E(n),\calo_X))$.}
\end{proof}

\begin{remark}
Note that whenever we are given a pair of l.c.m curve $C, \: C'$ whose sheaves of ideals fit $0\to\cali_C\to\cali_{C'}\to T_E(r_X-n)\to0$, we obtain a short exact sequence like the one in display \eqref{ses-canonical} and, taking global sections, a short exact sequence like the one in display \eqref{ses-canonical-gs}.
In particular, if $n\ge r_X+q_X$, then $H^0(\inext^2(T_E(n),\calo_X))=0$ (recall that $\inext^2(T,\calo_X(-r_X))$ is a rank 0 instanton) which means that there are no instanton bundles corresponding to the curve $C$ and every $\xi\in H^0(\omega_{C}(r_X+i_X-2n))$ corresponds to a non locally free instanton sheaf that is singular along $\supp(T_E)$.
The only cases in which a curve can correspond both to a locally free and and to a non locally free instanton sheaf occur therefore for $i_X=3,4$ and $n=1$. Since a curve $C$ os such a kind satisfies (cf. \ref{instantonic-curve}) $H^i(\cali_C(-1))=H^{i-1}(\calo_C(-1))=0, \ i=1,2$, from remark \ref{rmk:lines-arrangement}, we see that the only curves corresponding both locally free and non-locally free rank 2 instantons of charge $n$ are the lines arrangement of degree $n+1-r_X$ (degree of $C$ is computed applying \ref{genus-degree-inst}). An example of a family of instanton bundles corresponding to a degree $n$ line arrangement $C$ and degenerating to a non locally free instanton, still corresponding to $C$ was exhibited in example \ref{ex:family}.
\end{remark}

Next, we consider the reverse construction: let $(C',\xi')$ be a pair consisting of a l.c.i curve $C$ satisfying
$$ H^i(\cali_{C'}(n-r_X-q_X))=0, \ i=1,2, $$
and a nowhere vanishing section $\xi'\in H^0(\omega_{C'}(i_X+r_X-2n))$. Considering the short exact sequence
\begin{equation}\label{sqc:f}
0 \longrightarrow \ox(-n) \stackrel{r}{\longrightarrow} F \longrightarrow \cali_{C'}(n-r_X) \longrightarrow 0
\end{equation}
given by regarding $\xi'$ as a class in $\Ext^1(\cali_{C'}(n-r_X),\ox(-n))$ and the second part of Remark \ref{rem:five}, it follows that the rank 2 sheaf $F$ in the corresponding pair $(F(n),r)$ is a locally free instanton sheaf.

\begin{lemma}
Any pair $(C,\xi)$ consisting of
\begin{enumerate}
\item a l.c.m. curve $C$ containing $C'$ such that $T:=\cali_{C'}(n-r_X)/\cali_C(n-r_X)$ is a rank 0 instanton sheaf;
\item a section $\xi\in H^0(\omega_{C}(i_X+r_X-2n))$ lying in the image of the induced map
$$ H^0(\omega_{C'}(i_X+r_X-2n))\hookrightarrow H^0(\omega_C(i_X+r_X-2n)) $$
\end{enumerate}
corresponds to a pair $(E(n),s)$ where $E$ is a non locally free rank 2 instanton sheaf which is singular along $\supp(T)$.
\end{lemma}
\begin{proof}
The sequence in display \ref{sqc:f} and the way the curve $C$ is chosen provides us with the diagram
\begin{equation}\begin{split}\label{diag3}
\xymatrix{
& & 0\ar[d] & & \\
& & \ox(-n)\ar[d]^{r} & & \\
& & F \ar[rd]^{q}\ar[d] & & \\
0\ar[r] & \cali_{C}(n-r_X) \ar[r] & \cali_{C'}(n-r_X) \ar[r]\ar[d] & T \ar[r] & 0\\
& & 0 & &
}\end{split}\end{equation}
The sheaf $E:=\ker(q)$ satisfies two short exact sequences; the first one
$$ 0 \longrightarrow E \longrightarrow F \longrightarrow T \longrightarrow 0 $$
implies that $E$ is an instanton sheaf, while the second
$$ 0 \longrightarrow \ox(-n) \longrightarrow E \longrightarrow \cali_{C}(n-r_X) \longrightarrow 0 $$
induces the section $\xi\in H^0(\omega_C(i_X+r_X-2n))$ which vanishes on $\supp(T)$, and therefore lies in the image of the map given in the statement of the lemma.
\end{proof}

The reason why we chose to portray in detail how the Serre correspondence works for rank 2 instantons is simply due to the fact that we are mainly concerned with moduli spaces of rank 2 sheaves. Of course these arguments can be generalized to instantons of arbitrary rank. Doing so we can in particular provide examples of non-locally free reflexive instantons of rank $>2$ (we recall indeed that in rank 2 the reflexivity of instantons is equivalent to their local freeness).

\begin{example}
We can construct a non locally free reflexive instanton on $X$ being either $\p3$ or a quadric threefold as follows. Let $C\subset X$ be a smooth rational curve  of degree $d\ge 4$ and we consider two linearly independent sections $\xi_i\in H^0(\calo_C((d-2){\rm pt})),\ i=1,2$, whose zero loci intersect along a 0-dimensional scheme $Z\subset C$ of length $d'<d-2$. Since $\calo_C((d-2){\rm pt})\simeq\omega_C(1)$, these two sections correspond to an extension class in  $\Ext^1(\cali_C(2-r_X),\ox(-1)^{\oplus2})$, thus giving rise to a short exact sequence
\begin{equation}\label{rank3-refl}
0\rightarrow \ox(-1)^{\oplus 2}\rightarrow E\rightarrow \cali_C(2-r_X)\rightarrow 0;
\end{equation}
we argue that the middle term $E$ is the sheaf we are looking for. 

Indeed, the exact sequence in display \eqref{rank3-refl} yield 
$$ H^i(E(-q_X))\simeq H^i(\cali_C)=0, \ \forall\: i.$$
Dualizing the same exact sequence and recalling that $\inext^1(\cali_C(2-r_X),\ox)\simeq\omega_C(2)$, we obtain:
$$ 0\rightarrow \ox(r_X-2)\rightarrow E^{\vee}\rightarrow \ox(1)^{\oplus 2}\xrightarrow{\xi} \omega_C(2) \rightarrow \inext^1(E,\ox)\rightarrow 0$$
where the morphism $\xi$ is defined by the two sections 
$(\xi_1,\xi_2)\in H^0(\omega_C(1)^{\oplus 2})$ we started with. By construction, $\xi$ fails to be surjective along $Z$, so that $\inext^1(E,\calo_{X})$ is supported on $Z$.
This, together with the vanishing of $\inext^i(E,\calo_{X})$ for $i>1$ implies that $E$ is reflexive with $\sing(E)=Z$, thus non locally free.

Finally, to see that $E$ is $\mu$-semistable it is enough to check that $H^0(E(r_X-1))=H^0(E^{\vee}(r_X-1))=0$, cf. \cite{OSS}.
$H^0(E(r_X-1))$ vanishes since it is isomorphic to $H^0(\cali_C(1-2r_X))=0$ (this is clearly zero for $r_X=1$ whilst for $r_X=0$ it is ensured by the fact that $C$ is not a planar curve). Setting $F:=\ker(\xi)$, note that $H^0(E^{\vee}(r_X-1))=H^0(F(r_X-1))$, and that the latter coincides with the kernel of the induced map
$$ H^0(\ox(r_X)^{\oplus2}) \stackrel{(\xi_1,\xi_2)}{\longrightarrow} H^0(\omega_C(r_X+1)), $$
given by multiplication by the sections $\xi_i, \ i=1,2$.

When $r_X=0$ (ie. $X=\p3$) the fact that $\xi_i$ are linearly independent is enough to guarantee that this map is injective, thus $H^0(F(r_X-1))=0$, as desired.

If $r_X=1$,  one must argue that $(\xi_1,\xi_2)$ does have not a syzygy $(\sigma_1,\sigma_2)\in H^0(\calo_C(d~{\rm pt}))$ of degree $d$ that lies in the image of the restriction map
$$ H^0(\ox(1)^{\oplus2})\to H^0(\calo_C(1)^{\oplus2}) \simeq  H^0(\calo_C(d~{\rm pt})). $$
This seems to be a generic condition when $d-d'$ is sufficiently large, but we have not been able to prove it.
\end{example}


\section{Instanton sheaves on quadric threefolds}\label{sec:quadrics}

Let $V$ be a 5-dimensional vector space and consider a smooth quadric hypersurface $X\subset \D{P}(V)\simeq \p4$. $X$ is the only Fano 3-fold of Picard rank one and index 3, therefore an instanton sheaf $E$ on $X$ is defined as a torsion free $\mu$-semistable sheaf with $c_1(E)=-1$ and such that:
\begin{equation}\label{instantonic-quadric}
H^i(E(-1))=0, \ i=1,2.
\end{equation}
Recall that since $c_1(E)$ is odd, every instanton sheaf on $X$ is actually $\mu$-stable; this ensures the vanishing of $H^i(E(-1))$ for $i=0,3$ as well (cf. Lemma \ref{instantonic}).
From now on we will only be concerned with instanton sheaves of rank 2 (therefore when referring to an instanton sheaf we will always imply that its rank is 2).

The Chern character of a rank 2 instanton $E$ of charge $n$ is:
\begin{equation}\label{ch-char}
{\rm ch}(E)=\left(2, -[H],~ (1-n)[l],~ \frac{-1}{3} + \frac{n}{2}\right)
\end{equation}
(by Corollary \ref{c3-inst}, $c_3(E)=0$); applying Riemann--Roch, we compute the Hilbert polynomial of $E$: 
\begin{equation}\label{Hilbert}
P_{E}(t)=\frac{2}{3}t^3+2t^2+\left(\frac{7}{3}-n\right)t+ (1-n)
\end{equation}

In this section we present some results on instanton sheaves on $X$.
We will focus our attention on instanton sheaves $E$ of charge 2, emphasising the relation that these sheaves have with the curves corresponding to global sections of $E(1)$ via Serre correspondence.
The Serre correspondence allows us not only to describe the instanton moduli space, but to obtain also a complete picture of the entire Gieseker--Maruyama moduli space $\mathcal{M}:=\mathcal{M}_X(2,-1,2,0)$ of semistable rank 2 sheaves with Chern classes $(c_1,c_2,c_3)=(-1,2,0)$, together with its relation with the Hilbert scheme $\Hilb_{2t+2}(X)$.


\subsection{Instanton sheaves of charge 1}\label{generalities}

Since every instanton sheaf $E$ is a $\mu$-semistable sheaf with $c_1(E)=-1$, the Bogomolov inequality implies that $c_2(E)\ge1$. In the case $c_2=1$, a well known example of instanton bundle is provided by the so called \textit{spinor bundle} which will be henceforth denoted by $\cals$.

Recall that it can be defined as follows, cf. \cite[Definition 1.3]{O-quadrics} for which we refer to for all details in this paragraph. There is an embedding $s:X\to\Grp(1,3)$, the grassmannian of lines in $\p3$, see ; then $\cals:=s^*U$, where $U$ is the universal bundle on $\Grp(1,3)$. This is a $\mu$-stable rank 2 bundle on $X$ with $c_1(\cals)=-1$ and $c_2(\cals)=1$. In addition, $\cals$ is rigid \cite[Theorem 2.1]{O-quadrics}, and every $\mu$-stable rank 2 bundle $E$ on $X$ with $c_1(E)=-1$ and $c_2(E)=1$ is isomorphic to the spinor bundle $\cals$.

Since $h^0(\cals(1))=4$, Serre construction provides the following short exact sequence
\begin{equation} \label{defn-spinor}
0 \longrightarrow \ox(-1) \longrightarrow \cals \longrightarrow \cali_\ell \longrightarrow 0
\end{equation}
where $\ell$ is a line in $X$. It is then easy to see that $\cals$ is a rank 2 instanton bundle of charge 1. This observation allows us to give the following characterization of the family $F(X)$ of lines on $X$. Since $h^0(\cals)=0$ (by stability) we deduce then that $\forall s \in H^0(\cals(1)), \ s\ne 0$, $\coker(s)$ is torsion free and isomorphic to $\cali_{l}(1)$ for a line $l\subset X$ (this last assertion follows from a simple Chern character computation) . Conversely, $\forall\: [l]\in F(X)$, every sheaf fitting in a non-split short exact sequence of the form \eqref{defn-spinor} is a $\mu$-stable vector bundle with $c_1=-1,\ c_2=1$ and is therefore isomorphic to $\cals$. Accordingly $F(X)\simeq \D{P}^3\simeq \D{P}(H^0(\cals(1))$. 

\begin{proposition}
Every rank 2 instanton sheaf of charge $1$ on $X$ is isomorphic to the spinor bundle.
\end{proposition}
\begin{proof}
Let $E$ be a rank 2 instanton sheaf of charge $1$. If $E$ is reflexive, then it must be locally free and therefore it is isomorphic to the spinor bundle. 

If $E$ is not reflexive, then Theorem \ref{double-dual-bundle} implies that $E^{\vee\vee}$ is a locally free instanton sheaf of charge $c_2(E^{\vee\vee})\ge1$. However, $c_2(E^{\vee\vee})+\deg(T_E)=c_2(E)=1$, thus in fact $c_2(E^{\vee\vee})=1$ and $\deg(T_E)=0$. It follows that $T_E=0$, contradicting the hypothesis that $E$ was not reflexive.
\end{proof}

The following result will also be useful later on.

\begin{lemma}\label{refl-low}
Every $\mu$-stable rank 2 reflexive sheaf with Chern classes $c_1=-1$ and $c_2=1$ is isomorphic to the spinor bundle.
\end{lemma}
\begin{proof}
Let $F$ be a $\mu$-stable rank 2 reflexive sheaf with $c_1(F)=-1$ and $c_2(F)=1$, so that 
$$ \chi(F)=\dfrac{c_3(F)}{2} = h^2(F)-h^1(F) $$
since $h^0(F)=h^3(F)$ by $\mu$-stability. We claim that $h^1(F(n))=0, \ \forall n\le 0$.

Indeed, take a general hyperplane section $Q\in |\calo_X(1)|$ and consider the restriction sequence 
\begin{equation}\label{ses-restriction}
0\longrightarrow F(-1) \longrightarrow F \longrightarrow F|_{ Q} \longrightarrow 0,
\end{equation}
with $F|_Q$ being a $\mu$-semistable locally free sheaf on $Q$ with $c_1(F|_Q)=(-1,-1)$. The $\mu$-semistability of $F|_Q$ leads to the vanishing of $H^0(F|_Q(n)), \ \forall n\le 0$ and of $H^2(F|_Q(n)),\ \forall n\ge -1$. In addition, since $\chi(F|_Q)=1-c_2(F)=0$, we conclude that $h^1(F|_Q)=0$; Serre duality then implies that $h^1(F|_Q(-1))=0$. From the fact that $h^1(F|_Q)=h^2(F|_Q(-1))=0$, we deduce that $F|_Q$ is 1-regular which implies that $h^1(F|_Q(n))=0, \ \forall n\ge 0$. Since by Serre duality $h^1(F|_Q(n))=h^1(F|_Q(-n-1))$, we can finally conclude that $h^1(F|_Q(n))=0, \ \forall n\in \D{Z}$. Twisting the sequence in display \eqref{ses-restriction} and taking cohomology, we thus get $h^1(F(n))=h^1(F(n+1)), \ \forall n\le -1$; but from the reflexivity of $F$, $H^1(F(n))=0$ for $n\ll0$ hence $h^1(F(n))=0, \ \forall n\le 0$, as desired.

It follows that $h^2(F)=\dfrac{c_3(F)}{2}$; since $h^i(F|_Q(n))=0$ for $i=1,2$ and $\forall n\ge -1$ we also get, inductively, that $h^2(F(n))=h^2(F(n+1))$ $\forall n\ge -2$ so that $h^2(F(n))=\dfrac{c_3(F)}{2}, \ \forall n\ge -2$. By the Serre vanishing theorem, we must have that $h^2(F(n))=0$ when $n\gg0$, thus in fact $h^2(F(n))=0$ for every $n\ge-2$, and hence $c_3(F)=0$, implying that $F$ must be locally free. But every rank 2 locally free sheaves with $c_1=-1$ and $c_2=1$ on $X$ is a spinor bundle.
\end{proof}

We end this preliminary section summoning some properties of $F(X)$, the family of lines on $X$.
We have already recalled that $F(X)\simeq \D{P}^3\simeq \D{P}(H^0(\cals(1))$.
One "geometric" way to realize $F(X)$ as $\p3$ is the following. We start by constructing $X$ as a hyperplane section of the Grassmannian $\Grp(1,3)\subset \p5$ of lines in $\p3$.
Recall now that we have 2 families of planes contained in $\Grp(1,3)$: we have planes corresponding to families of lines passing trough a point (we call them planes of type I), and planes parameterizing families of lines contained in a plane $\D{P}^2\subset \p3$ (these will be referred to as planes of type II).
For each line $l\subset X$ there exists a unique pair of planes $(\Delta_I, \ \Delta_{II})$ with $\Delta_I$ of type $I$, $\Delta_{II}$ of type $II$, containing $l$; these planes are both parameterized by a 3-dimensional linear space $\p3$.

Several of our next results will rely on the geometry of linear spaces of lines; for this reason we recall here briefly some of their fundamental property.
We have two families of pencils of lines in $F(X)$.
Consider indeed a pencil $\p1\subset F(X)$ and denote by $l_0, \ l_1$ a pair of generators. If ever $l_0\cap l_1=\emptyset$, then the entire $\p1$ is a ruling in the quadric surface $Q:=\langle l_0,l_1\rangle \cap X$. In particular we must have that \textit{any} pair of lines in $\p1$ are disjoint hence $Q$ must be smooth since we have no disjoint lines in a singular hyperplane section of $X$.
For the same reason we deduce that if ever $l_0\cap l_1\ne \emptyset$, then any pair of lines in $\p1$ must intersect so that this family must coincide we the family of lines on a singular hyperplane section of $X$.

This implies in particular that these lines all have the form $\overline{qp}$ with $p$ fixed and $q$ varying along a conic. 
From these observations we deduce that there exists a morphism
$$ \Grp(1,F(X))\xrightarrow{\gamma} \p4^{*}$$
and that moreover, denoting by $\p4^{*}_{sm}:=\p4^{*}\setminus X^{*}$ the open of smooth hyperplane sections and by $\calu:= \gamma^{-1}( \p4^{*}_{sm})$, $\gamma|_{\calu}$ is a degree 2 covering over $\p4^{*}_{sm}$.

Finally we recall that $\forall \ l\subset X$ we have a hyperplane
$\p2\subset F(X)$ of lines meeting $l$
(isomorphic to the family of planes in $\p4$ containing $l$) and that all the hyperplanes in $F(X)$ are of this form.


\subsection{Instantons of charge 2}\label{sect-inst}

Our study of the moduli space $\calm$ starts with the study of $\call(2)$, the open subscheme parameterizing rank 2 instanton sheaves of charge 2. We will prove the following:

\begin{theorem}\label{instanton-component}
$\call(2)$ is a smooth, irreducible, 6-dimensional open subscheme of $\calm$ whose general element is a locally free instanton sheaf. Its closure $\overline{\call(2)}$ is an irreducible component of $\calm$. The moduli space $\calm$ is smooth along $\call(2)$. 
\end{theorem}
The moduli space $\cali(2)$ of locally free instanton sheaves of charge 2 was studied in \cite{Ott-Szu}. In loc. cit the authors proved the following.

\begin{theorem}\label{moduli-bundle}\cite[Theorem 4.1]{Ott-Szu}
The moduli space $\cali(2)$ is locally a trivial algebraic fibration over $(\p4)^{*}_{sm}$ with fibre being two disjoint copies of $\p2\setminus C_2$, for a smooth conic $C_2$. In particular it is a Stein manifold of dimension 6, rational irreducible and smooth.
\end{theorem}

The key ingredient of this result is the description of the families of curves arising as zero loci of global sections of $E(1)$ for $[E]\in \cali(2)$.

\begin{proposition}\label{curve-bundle}\cite[Proposition 4.4]{Ott-Szu}
The zero set $V(s)$ of a global section $s$ of $E(1)$ is a divisor of type $(2,0)$ on a smooth hyperplane section $Q\subset X$ (and hence it is either the union of two disjoint lines or a double line of arithmetic genus -1). The zero sets $V(s), \ V(t)$ of two sections $s,t$ of $E(1)$ lie on the same smooth quadric $Q$ and cut a system $g_2^1$ without base point.
\end{proposition}
This characterisation of the linear spaces $\D{P}(H^0(E(1)))$ implies indeed the existence of a morphism:
$$ \cali(2)\xrightarrow{\phi} \p4^{*}_{sm},$$
mapping a point $[E]\in \cali(2)$ to the quadric surface containing all the curves $V(s), \ s\in H^0(E(1))$. The fibre of $\phi$ over $Q$ consist of the base point free pencils of divisors of type $(2,0)$ or $(0,2)$ on $Q$, namely of two copies of $\p2\setminus C_2$ where $C_2$ is a smooth conic. 
The pencils of divisors of type $(2,0)$ are indeed parameterized by the projective space $\Grp(1,|\calo_Q(2,0)|)\simeq |\calo_Q(2,0)|^{*}\simeq \p2$;
inside this projective space, the locus of pencils with a base point identifies with $C_2$, the smooth conic of lines tangent to 
$\Gamma_2\subset |\calo_Q(2,0)|$, the conic parameterizing double lines. 

\begin{remark}\label{factor-Grass}
Note that, by construction, the morphism $\phi$ factors through a morphism $\phi_{\calu}:\cali(2)\rightarrow \calu$ where we recall that $\mathcal{U}\subset{\Grp(1,F(X))}$ is defined as the open subset parameterizing rulings of smooth hyperplane sections of $X$. 
\end{remark}

We now pass to the study of non locally free instantons $[E]\in \mathcal{M}$. By Theorem \ref{double-dual-bundle}, if $E$ is a non locally free instanton, $E^{\vee\vee}$ is an instanton bundle of charge $c_2(E^{\vee\vee})<c_2(E)$ and $T_E:=E^{\vee\vee}/E$ is a rank 0 instanton of degree $c_2(E)-c_2(E^{\vee\vee})$. Since $c_2(E)=2$ and the minimal charge of an instanton sheaf on $X$ is 1, 
the only possibility is that $E^{\vee\vee}\simeq \Ss$ so that $T_E$ is a rank 0 instanton of degree 1. It is not difficult to prove that $T_E\simeq \calo_l$ for a line $l\subset X$. Since $P_{T_E}(t)=t+1$ and as $h^1(T_E(-1))$ implies $h^1(T_E)=0$, we have $h^0(T_E)=1$. For $s\in H^0(T_E)$, the image $\im(s)$ of the corresponding morphism $\calo_X\xrightarrow{s}T_E$ must therefore be of the form $\calo_C$ for $C$ a degree one l.c.m curve. But this means $C\simeq l$ for a line $l\subset X$ and since $P_{\calo_l}=P_{T_E}$ we conclude that $T_E\simeq \calo_l$.  
Summing up, each non locally free instanton $E$ of charge 2 is defined by a short exact sequence of the form:
\begin{equation}\label{ses-sing-line}
    0\rightarrow E\rightarrow \Ss\xrightarrow{q}\calo_l\rightarrow 0.
\end{equation}

Our next aim is to formulate results similar to Proposition \ref{curve-bundle} and Theorem \ref{moduli-bundle} for non locally free instanton sheaf. We start describing the families of curves corresponding to global sections of $E(1)$.

\begin{proposition}\label{curves-singular}
Let $E$ be a non locally free instanton of charge 2 singular along a line $l$. Then $H^0(E(1))\simeq \D{C}^2$ and $\forall s\in H^0(E(1)), \ s\ne 0$, $\coker(s)\simeq \cali_{l'\cup l}$ with $l'$ varying in a ruling of a smooth hyperplane section of $X$ containing $l$.
\end{proposition}

\begin{proof}
Let us start with the computation of $H^0(E(1))$. Twisting (\ref{ses-sing-line}) and taking global sections, we obtain a linear map $H^0(\Ss(1))\rightarrow H^0(\calo_l(1))$ that can not be injective, (as $h^0(\Ss(1))=4$) hence $H^0(E(1))\ne 0$.
Denote now by $\iota$ the inclusion $\iota:H^0(E(1))\hookrightarrow H^0(\Ss(1))$. Since every non-zero element in $H^0(\Ss(1))$ has torsion free cokernel, the same holds for any non-zero $s\in H^0(E(1))$; this implies that $\coker(s)\simeq \cali_Y(1)$ for a l.c.m subscheme $Y\subset X$.

For any $s\in H^0(E(1)),\: s\ne 0$, we have $\coker(\iota(s))\simeq \cali_{l'}(1)$ for a line $l'\subset X$ and we get a commutative diagram:
\begin{equation}\label{cd-sing-inst}
\begin{tikzcd}
0  \arrow[r] &\OO_X\arrow[r]\arrow[d, "s"] &\OO_X\arrow[d, "\iota(s)"]\arrow[r] & 0 \arrow[d] \arrow [r] &0 \\
0  \arrow[r] & E(1)\arrow[r]\arrow[d, two heads] & \Ss(1)\arrow[r, "q"]\arrow[d, two heads] & \OO_{l}(1) \arrow [r]\arrow[d, "id"]& 0\\
0\arrow[r] &\cali_Y(1)\arrow[r] & \cali_{l'}(1)\arrow[r]& \calo_l(1) \arrow[r]
& 0\\
\end{tikzcd}
\end{equation}

From it we compute that $P_Y(n)=2n+2$ and we deduce that $\supp(Y)=l'\cup l$. Since $Y$ must be l.c.m. and as $\cali_{l'}$ surjects onto $\calo_l$, the only possibilities are either $l'=l$, in which case $\cali_l|_l\simeq \calo_l\oplus \calo_l(-1)$ and $Y$ is a double structure on $l$ with arithmetic genus $-1$ and , or $l'\cap l=\emptyset$ in which case $\cali_{l'}|_l\simeq \calo_l$ and $Y$ is simply the union of $l$ and $l'$. In each of these cases, the scheme $Y$ is contained in the unique hyperplane $\langle Y\rangle \simeq \p3$ thus, from the first column of (\ref{cd-sing-inst})
we compute that $h^0(E(1))=2$. To complete the proof of the proposition we still need to describe the pencil $\D{P}(H^0(E(1)))$. By construction the space of section $r_l\in H^0(\cals(1))$ vanishing on $l$ locates a point in $\D{P}(\iota(H^0(E(1)))\simeq \p1$ (since $r_l\otimes \calo_l$=0); 
the arguments previously presented show that a generic element in $\iota(H^0(E(1)))$ corresponds to a line disjoint from $l$.
From the discussion held at the end of section \ref{generalities}, the pencil $\D{P}(\iota(H^0(E(1))))$ must therefore coincide with the ruling of a smooth hyperplane section $Q$ of $X$ containing $l$ and the curves corresponding to non zero sections of $E(1)$ are thus all of the form $l\cup l'$ with $l'$ varying in $\D{P}(\iota(H^0(E(1))))$.
\end{proof}

\begin{remark}
Proposition \ref{curves-singular} allows us to deduce that, given a l.c.m. curve $Y$ with Hilbert polynomial $2t+2$ and $\supp(Y)=l\cup l'$, a pair $(s,E(1))$ with $E$ a non locally free instanton singular along $l$, corresponds to a pair $(Y,\xi)$ with $\xi\in H^0(\omega_Y(2))\simeq H^0(\calo_l)\oplus H^0(\calo_l')$ of the form $(0,e), \ e\ne 0.$  
\end{remark}
We now want to understand how non locally free instantons behave in families. We consider therefore the set $\cald(1,1):=\call(2)\setminus \cali(2)$ that parameterizes non locally free instantons.
\begin{proposition}\label{moduli-singular}
$\cald(1,1)$ is a locally closed subscheme of $\calm$; it is smooth, irreducible and of dimension 5.
\end{proposition}

\begin{proof}
By semicontinuity, the instanton locus $\call(2)$ is open in $\calm$; since being non locally free is a closed condition, we have that $\cald(1,1)$ is locally closed in $\calm$. To prove the rest of the proposition we mimic the proof of Theorem \ref{moduli-bundle}.
From Proposition \ref{curves-singular} we know that for $[E]\in \cald(1,1)$, the curves in the pencil $\D{P}(H^0(E(1)))\subset \Hilb_{2t+2}(X)$
are of the form $l\cup l'$ with $l=\sing(E)$ and with the $l'$s varying in a ruling of a smooth hyperplane section of $X$. 

As it was the case for $\cali(2),$ $\cald(1,1)$ is endowed as well with a surjective map $\cald(1,1)\xrightarrow{\phi^{(1,1)}} \p4^{*}_{sm}
$ fitting in a commutative diagram
\begin{equation*}
  \xymatrix@C+1em@R+1em{
   \cald(1,1) \ar[r]^-{\phi^{(1,1)}} \ar[d]_-{\phi^{(1,1)}_{\calu}} & \p4^{*}_{sm} \\
   \mathcal{U} \ar[ur]
  }
\end{equation*}
but this time the fibre over a ruling $|\calo_Q(1,0)|\simeq \p1\in \calu, \ [Q]\in (\p4)^{*}_{sm}$ consists of pencils of divisors of type $(2,0)$ with a base point. Each fibre of $\phi^{(1,1)}_{\calu}$ is therefore isomorphic to $C_2\subset \Grp(1,|\calo_Q(2,0)|)$; the smooth conic parameterizing the tangents to the locus of singular divisors in $|\calo_Q(2,0)|$.
This proves that $\cald(1,1)$ is smooth irreducible and of dimension equal to five.
\end{proof}

\begin{remark}\label{pencil-emb}
Applying arguments equivalent to \cite{H-bundle} Lemma 9.3., we can make the following considerations. Let $[E]$ be a point corresponding to an instanton and let us consider the short exact sequence
$$ 0\rightarrow \calo_X\xrightarrow{s} E(1)\rightarrow \cali_Y(1)\rightarrow 0$$
induced by $s\in H^0(E(1))$. From this short exact sequence we deduce that the image of $t\in H^0(E(1))$ in $H^0(\cali_Y(1))$ gives an equation for the hyperplane $\langle Y\rangle$ and that moreover for any $t\in H^0(E(1))$ independent from $s$, $E(1)/(s,t)\simeq \cali_{Y,Q}(1)$, for \mbox{$Q:=X\cap \langle Y\rangle$.}
We have therefore a well defined linear map \mbox{$H^0(E(1))\rightarrow H^0(\calo_Q(2,0))$} that maps each $s\in H^0(E(1))$ to the form defining $V(s)$ on $Q$. 
\end{remark}

\subsection{A description of $\call(2)$ via Serre correspondence.}

For the moment we just know that $\cali(2)$ and $\cald(1,1)$ are locally trivial fibrations over $\calu$. Using Serre correspondence we are now going to show that actually, the entire $\call(2)$ 
identifies with a $\D{P}^2$ bundle over $\calu$ and that $\cald(1,1)$ and $\cali(2)$ are, respectively, a divisor and an open subset of $\call(2)$. The key ingredient to prove this is the Serre correspondence which enables us to collect information about the geometry of $\call(2)$ studying the geometric properties of the families of the corresponding curves.
Our starting point is therefore the inspection of the open 
$\calh\subset \Hilb_{2t+2}(X)$ that parameterises locally Cohen Macaulay curves. 
Note that any locally Cohen Macaulay curve with Hilbert's polynomial $2t+2$ is indeed either the union of two disjoint lines or a double structure on a line of arithmetic genus -1.
\begin{lemma}\label{smooth-l.c.m.}
$\forall\: [Y]\in \calh$, $h^0(\caln_{Y/X})=6$ and $h^1(\caln_{Y/X})=0$.
\end{lemma}

\begin{proof}
We first show that for $[Y]\in \calh$, $Y$ lies in a unique smooth hyperplane section $Q\subset X$.

If $Y=l_1\cup l_2, \ l_1\cap l_2=\emptyset$ then the only hyperplane containing $Y$ is $\langle Y\rangle=\langle l_1,l_2\rangle$. 
If otherwise $Y$ is a double line supported on $l$, we have that $\cali_Y$ fits in the exact sequence
$$ 0\rightarrow \cali_Y\rightarrow \cali_l\rightarrow \calo_l\rightarrow 0 $$
from which we compute that $h^0(\cali_Y(1))\ne 0$ (since $H^0(\cali_l(1))$ can not inject in $H^0(\calo_l(1))$) and $h^0(\cali_Y(1))<2$ (since no planar l.c.m. curve has negative arithmetic genus). Thus $h^0(\cali_Y(1))=1$ and $Y$ is contained in a unique hyperplane section $Q$ of $X$. $Q$ must be smooth since a degree 2 l.c.m. curves on a singular hyperplane section of $X$ is planar. 
The only degree 2 l.c.m. curve on a cone of vertex $p$ over a smooth conic $C$ indeed are $C$ itself or cones degree 2 divisors $D\subset C$ of $C$.
These latter are always contained in a plane $\langle p, L_D\rangle$ with $L_D\subset \langle C\rangle$ the unique line spanned by $D$ for $D$ reduced whilst $L_D=\D{T}_qC$ for $D_{\rm red}=q$. 

Let us now compute $h^i(\caln_{Y/X}), \ i=0,1$.
Consider the smooth quadric surface $Q:=\langle Y\rangle \cap X$ and denote by $L_A, \ L_B$ the two generators of $\Pic(Q)$. The only degree 2 effective divisors in $Q$ having arithmetic genus -1 belong either to the class $2L_A$ or to $2L_B$.

Without loss of generality, we suppose then $Y\sim 2L_A$ and we consider $$0\rightarrow \OO_Q\rightarrow\OO_Q(2 L_A)\rightarrow \OO_Y(2L_A)\rightarrow 0.$$
From this short exact sequence,  since $H^i(\OO_Q)=H^i(\OO_Q(L_A))=0, \ \forall\: i\ge 1$, we compute  $h^0(\OO_Y(2L_A))=2$ and $h^1(\OO_Y(2L_A))=h^1(\mathcal{N}_{Y/Q})=0.$
 
We finally consider:
$$0\rightarrow \mathcal{N}_{Y/Q}\rightarrow \mathcal{N}_{Y/X}\rightarrow {\mathcal{N}_{Q/X}|}_{Y}\rightarrow 0.$$
${\mathcal{N}_{Q/X}|}_Y$ is isomorphic to $\OO_Y(1)$ and from
$$0\rightarrow \OO_l(1)\rightarrow \OO_Y(1)\rightarrow \OO_l(1)\rightarrow 0$$
we obtain $h^0(\OO_Y(1))=4$ and $h^1(\OO_Y(1))=0.$

From these arguments we deduce the vanishing of $H^1(\mathcal{N}_{Y/X})$, which implies the smoothness of $\Hilb_{2t+2}(X)$, and we compute that $\Hilb_{2t+2}(X)$ has dimension $6=h^0(\caln_{Y/X})=h^0(\calo_Y(1))+h^0(O_Y(2L_A))$ at $[Y]$.
\end{proof}

\begin{lemma}\label{HS-lcm}
$\calh$ is a $\D{P}^2$ bundle over $\calu\subset \Grp(1,F(X))$.
\end{lemma}
\begin{proof}
Denote by $\calt_{\calu}$ the restriction of the tautological bundle over $\Grp(1,F(X))$ to $\calu$.
Take then the rank 3 vector bundle $\Sym^2(\calt_{\calu}^{\vee})$ and the projective bundle $$\D{P}(\Sym^2(\calt_{\calu}^{\vee}))\rightarrow\calu.$$
For a point $h\in \D{P}(\Sym^2(\calt_{\calu}^{\vee}))$ we denote by $l_{1,h}, \ l_{2,h}$ the corresponding (possibly coincident) lines. The incidence correspondence $\Sigma\subset \D{P}(\Sym^2(\calt_{\calu}^{\vee}))\times X$
$$\Sigma:=\{(h,p)\in \D{P}(\Sym^2(\calt_{\calu}^{\vee}))\times X\mid p\in l_{1,h}\cup l_{2,h}\}$$
induces a bijective morphism $\D{P}(\Sym^2(\calt_{\calu}^{\vee})) \to \calh$ 
that, since $\calh$ and $\D{P}(\Sym^2(\calt_{\calu}^{\vee}))$ are smooth,
is therefore an isomorphism (this is due to Zariski main theorem).
\end{proof}

From now on we denote by $\pi_{\calh}:\calh\to \calu$ the standard projection.

Let us now pass to the study of $\call(2)$. To begin with we show how, from the smoothness of $\calh$ we can deduce the smoothness of $\calm$ along $\call(2)$.

\begin{lemma}\label{smooth-inst-bundle}
For any $[E]\in \cali(2), \ \ext^1(E,E)=6$ and $\ext^2(E,E)=0$.
\end{lemma}
\begin{proof}
Consider the short exact sequence:
\begin{equation}\label{ses-curve-bundle}
    0\rightarrow \calo_X(-1)\rightarrow E\rightarrow \cali_Y\rightarrow 0. 
\end{equation}
Applying $\Hom(E, \ \cdot \ )$ we obtain an exact sequence of vector spaces
$$ \Ext^1(E,\calo_X(-1))\rightarrow \Ext^2(E,E)\rightarrow \Ext^2(E,\cali_Y).$$
The left side term is zero since it is dual to $H^1(E(-2))\simeq H^1(\cali_Y(-2))=0$.
Let us now prove that the right side term vanishes as well. 
By stability and by Lemma \ref{cohom-interm}, we have $h^i(E(1))=h^{3-i}(E(-3))=0$ for $i=2,3$; as moreover $P_E(1)=2$ and $h^0(E(1))=2$, we conclude that $h^1(E(1))=0$ as well.
This implies the vanishing of $\Ext^i(E,\calo_X)\simeq \Ext^i(\calo_X,E(1))$ for $i=1,2$ which leads to an isomorphism $\Ext^2(E,\cali_Y)\simeq \Ext^1(E,\calo_Y)$.
But $E$ is locally free therefore: 
$$\Ext^1(E,\calo_Y)\simeq H^1(\inhom(E,\calo_Y))\simeq H^1(E|_Y^{\vee})\simeq H^1(\caln_{Y/X})=0$$
(the isomorphism $E|_Y^{\vee}\simeq \caln_{Y/X}$ is obtained tensoring (\ref{ses-curve-bundle}) for $\calo_Y$ and the vanishing of $H^1(\caln_{Y/X})$ is due to Lemma (\ref{smooth-l.c.m.})).
Therefore $\Ext^2(E,\cali_Y)=0$ which implies $\Ext^2(E,E)=0$. 
Now, the stability of $E$ leads to $\Hom(E,E)\simeq \D{C}$ and $\Ext^3(E,E)\simeq \Hom(E,E(-3))^*=0$. Since $E$ has homological dimension one, we can apply an argument equivalent to \cite[Proposition 3.4]{H-reflexive}, obtaining:
\begin{equation}\label{euler-char}
\chi(E,E)= \frac{3}{2}c_1(E)^2-6c_2(E)+4=-5.
\end{equation}

which allows to conclude that $\ext^1(E,E)=6$. This ensures that the moduli space $\calm$ is smooth along $\cali(2)$ and that  $\overline{\cali(2)}$ is the only component passing trough any point in $\cali(2)$.
\end{proof}

We pass now to the case of non locally free instantons.

\begin{proposition}\label{smooth-inst-sing}
$\calm$ is smooth of dimension 6 at any point $[E]\in \cald(1,1)$.
\end{proposition}
\begin{proof}
We know that $E$ fits in a short exact sequence:
\begin{equation}\label{ses-dd-line}
    0\rightarrow E\rightarrow \Ss\rightarrow \calo_l\rightarrow 0
\end{equation}
where $E^{\vee\vee}\simeq \mathcal{S}$ and $l=Sing(E).$
Applying $\Hom(\:\cdot\:, E)$ we end up with a sequence of vector spaces:
$$\Ext^2(\Ss,E)\rightarrow \Ext^2(E,E)\rightarrow \Ext^3(\calo_l, E);$$
$\Ext^3(\calo_l, E)\simeq \Hom(E,\calo_l(-3))\simeq \Hom(E\restriction_l, \calo_l(-3))$.
Tensoring (\ref{ses-dd-line}) for $\otimes\: \calo_l$ we obtain:
    $$ 0\rightarrow \intor_{1}(\calo_l,\calo_l)\xrightarrow{\alpha} E\restriction_l\xrightarrow{\beta}\Ss\restriction_l\xrightarrow{\gamma} \calo_l\rightarrow 0$$
and consequently:
\begin{equation}\label{ses-1}
0\rightarrow \intor_1(\OO_l,\OO_l)\xrightarrow{\alpha} E\restriction_l\xrightarrow{\beta} Im(\beta)\rightarrow 0,
\end{equation}
\begin{equation}\label{ses-2}
0\rightarrow Im (\beta)\rightarrow \Ss\restriction_l\xrightarrow{\gamma} \OO_l\rightarrow 0.
\end{equation}
$\intor_{1}(\OO_l,\OO_l)\simeq \mathcal{N}_{l/X}^{\vee}\simeq \OO_l\oplus \OO_l(-1)$, thus $\Hom(\intor_{1}(\OO_l,\OO_l), \OO_l(-3))=0$ so that $\Hom(E\restriction_l, \OO_l(-3))\simeq \Hom(Im(\beta), \OO_l(-3)).$ From (\ref{ses-2}) $Im(\beta)$ is a rank one torsion free sheaf of degree -1; but $l$ is a line, therefore $Im(\beta)$ is a line bundle of degree -1, so that $Im(\beta)\simeq \OO_l(-1)$. From this we deduce $\Hom(Im(\beta), \OO_l(-3))=0$ and consequently that $\Hom(E\restriction_l,\OO_l(-3))\simeq \Ext^3(\OO_l,E)=0$.

To prove the vanishing of $\Ext^2(\Ss, E)$, we apply $\Hom(\Ss, \cdot)$ to (\ref{ses-dd-line}), getting:
$$\Ext^1(\Ss,\OO_l)\rightarrow \Ext^2(\Ss,E)\rightarrow \Ext^2(\Ss,\Ss)=0.$$
Since $\Ss$ is locally free, $\Ext^1(\Ss, \OO_l)\simeq H^1(\inhom(\Ss,\OO_l))$ and this latter vanishes again due to $\Ss\restriction_{l}\simeq \OO_{l}\oplus \OO_{l}(-1)$.
These computations yield to $\Ext^2(E,E)=0$ implying the smoothness of $\mathcal{M}$ at $E$.
Also this time the stability of $E$ ensures that $\hom(E,E)=1$ and $\ext^3(E,E)=0$, and since the homological dimension of $E$ is one, we can again argue as in \cite[Proposition 3.4]{H-reflexive}, which leads to $\chi(E,E)= \frac{3}{2}c_1(E)^2-6c_2(E)+4=-5$. This implies that $\ext^1(E,E)=6$ ending our proof.
\end{proof}

We consider now the scheme $\calb$ parameterizing the pencils of curves \mbox{$\D{P}(H^0(E(1))$} for $[E]\in \call(2)$.
$\calb$ identifies with the Grassmann bundle:
\begin{equation}\label{defn calb}
\calb:=G_2(\Sym^2(\calt_{\calu}^{\vee}))\simeq \D{P}(\Sym^2(\calt_{\calu}^{\vee})^{\vee})
\xrightarrow{\pi_{\calb}}\calu.
\end{equation}

By construction $\calb$ is a smooth and irreducible 6-dimensional variety. Our next goal is to show that $\calb$ is isomorphic to $\call(2)$. To prove this we will construct a projective bundle $\D{P}(\cale)\to \calb$ that carries a family of instantons and such that the induced morphism $\D{P}(\cale)\to \call(2)$ factors trough an isomorphism $\calb\to \call(2)$. 

We start by considering the universal curve $\mathbf{Y}\subset \calh\times X$ and the relative ext sheaf \mbox{$\cale:=\inext^1_{p_1}(\cali_{\mathbf{Y}}(1), \calo_{\calh\times X})\in Coh(\calh)$}, 
where $p_1$ is the projection onto the first factor (here for $\calf\in Coh(\calh\times X)$ we define $\calf(n):=\calf\otimes {p_2}^*\calo_X(n)$.) 

\begin{proposition}\label{bundle-sc}
$\cale$ is a rank 2 vector bundle on $\calh$ and the projective bundle 
$\D{P}(\cale)$ admits the structure of a $\p1$ bundle over $\calb$.
\end{proposition}
\begin{proof}
Recall that $\cale:=R^i({p_1}_*\inhom(\cali_{\mathbf{Y}}(1), \ \cdot\ )) (\calo_{\calh\times X})$ hence, from the spectral sequence $R^p{p_1}_*(\inext^q(\cali_{\mathbf{Y}}(1),\calo_{\calh\times X}))\Rightarrow \inext^{p+q}_{p_1}(\cali_{\mathbf{Y}}(1),\calo_{\calh\times X})$
we obtain an exact sequence:
\begin{align*}
    0\rightarrow R^1{p_1}_*(\inhom(\cali_{\mathbf{Y}}(1),\calo_{\calh\times X}))\rightarrow \cale&\rightarrow\\
    \rightarrow  {p_1}_*(\inext^1(\cali_{\mathbf{Y}}(1),\calo_{\calh\times X}))\rightarrow R^2{p_1}_*(\inhom(\cali_{\mathbf{Y}}(1)&,\calo_{\calh\times X}))\rightarrow 0.
\end{align*}
As $R^i{p_1}_*(\inhom(\cali_{\mathbf{Y}}(1),\calo_{\calh\times X}))=0, \ i=1,2$,
we get $$\cale\simeq {p_1}_*(\inext^1(\cali_{\mathbf{Y}}(1),\calo_{\calh\times X}))\simeq {p_1}_*(\tilde{\omega}_{\mathbf{Y}}(2)),$$ where $\tilde{\omega}_{\mathbf{Y}}$ is the relative dualizing sheaf.
$\forall\: [Y]\in \calh,\: h^0(\omega_Y(2))=2$ and since $\calh$ is integral, we can conclude that $\cale$ is a rank 2 vector bundle.
    
The isomorphism $\cale\simeq {p_1}_*(\tilde{\omega}_{\mathbf{Y}}(2))$ also implies that $\cale$ commutes with base change,
and since ${p_1}_*\inhom(\cali_{\mathbf{Y}}(1),\calo_{\calh\times X})=0$, from \cite[Corollary 4.5]{universal-ext} we get the existence of a universal extension on $\D{P}(\cale)\times X$:
\begin{equation}\label{univ-ext}
0\rightarrow \calo_{\D{P}(\cale)\times X}\otimes {p_1}^*\calo_{\D{P}(\cale)}(1)\rightarrow \hat{\mathbf{E}}\rightarrow \cali_{\hat{\mathbf{Y}}}\rightarrow 0
\end{equation} 
where $\hat{\mathbf{Y}}\subset \D{P}(\cale)\times X$ is the pullback of the universal curve $\mathbf{Y}$.
Twisting and applying the functor ${p_1}_*$ we obtain a short exact sequence of vector bundles on $\D{P}(\cale)$:
\begin{equation}\label{univ-ex-pf}
0\rightarrow \calo_{\D{P}(\cale)}(1)\rightarrow {p_1}_*(\hat{\mathbf{E}}(1))\rightarrow \pi_{\cale}^*({p_1}_*(\cali_{\mathbf{Y}}(1)))\rightarrow 0
\end{equation}
($\pi_{\cale}:\D{P}(\cale)\to \calh$ is the standard projection).
We claim that the rank 2 vector bundle ${p_1}_*(\hat{\mathbf{E}}(1))$ on $\D{P}(\cale)$ induces a morphism $\D{P}(\cale)\to \calb$. This is obtained via a relative version of the argument presented in (\ref{pencil-emb}). 

Take an affine cover $V_i=Spec(A_i)$ of $\D{P}(\cale)$ that is trivialising for ${p_1}_*(\hat{\mathbf{E}}(1))$; on each $V_i$ we have $${p_1}_*(\hat{\mathbf{E}}(1)))|_{V_i}\simeq H^0(\hat{\mathbf{E}}(1)|_{V_i\times X})^{\sim}\simeq (A_i^2)^{\sim}.$$
 
For any non zero $s_i\in H^0(\hat{\mathbf{E}}(1)|_{V_i\times X})$, its image in $H^0(\cali_{\hat{\mathbf{Y}}}(1)|_{V_i\times X})$ determines a family $\hat{\mathbf{Q}}_i$ of hyperplane sections whose fibre over $v\in V_i$ is $\langle\hat{\mathbf{Y}}_{v}\rangle\cap X$. Moreover for any pair $s_i,t_i$ of generators of $H^0(\hat{\mathbf{E}}(1)|_{V_i\times X}),$ we have isomorphisms $(\hat{\mathbf{E}}(1)|_{V_i\times X})/(s_i,t_i)\simeq \cali_{\hat{\mathbf{Y}}_i, \hat{\mathbf{Q}}_i}(1)$ ($\hat{\mathbf{Y}}_i$ being the restriction of $\hat{\mathbf{Y}}$ to $V_i\times X$). We get therefore injective $A_i$-linear maps ${p_1}_*(\hat{\mathbf{E}}(1))|_{V_i}\hookrightarrow (\pi_{\calh}\circ \pi_{\cale})^*(\Sym^2(\calt_{\calu}^{\vee})|_{V_i})$
that glue defining an injective morphism ${p_1}_*(\hat{\mathbf{E}}(1))\hookrightarrow (\pi_{\calh}\circ \pi_{\cale})^*\Sym^2(\calt_{\calu}^{\vee})$. By the universal property of $\calb$, $\pi_{\calh}\circ \pi_{\cale}$ factors therefore trough a morphism $\rho:\D{P}(\cale)\to \calb$.

We finally show that $\D{P}(\cale)\xrightarrow{\rho}\calb$ is a projective bundle. Denote by $\calt_{\calb}\subset \pi_{\calb}^*(\Sym^2(\calt_{\calu}^{\vee}))$ the tautological rank 2 sub-bundle. By construction $\rho^*(\calt_{\calb})\simeq {p_1}_*(\hat{\mathbf{E}}(1))$ and since $\calo_{\D{P}(\cale)}(1)\hookrightarrow {p_1}_*(\hat{\mathbf{E}}(1))$,
we get that $\rho$ factors through a morphism $\D{P}(\cale)\xrightarrow{\rho'}\D{P}(\calt_{\calb})$ such that ${\rho'}^*(\calo_{\D{P}(\calt_{\calb})}(-1))\simeq \calo_{\D{P}(\cale)}(1)$. $\rho'$ is the morphism mapping a point $(Y,e), \ [Y]\in \calh, \ e\in \Ext^1(\cali_{Y}(1),\calo_X)$ to $([\D{P}(H^0(E(1)))], Y)\in \D{P}(\calt_\calb)$, $E=\hat{\mathbf{E}}_{(Y,e)}$.
$\rho'$ is a bijective morphism between smooth varieties, therefore it is an isomorphism. 
\end{proof}
\begin{remark}\label{bundle-double}
The variety $\D{P}(\cale)\simeq \D{P}(\calt_{\calb})$ identifies with the following incidence variety
$$ \D{P}(\cale)\simeq \D{P}(\calt_{\calb})\simeq \{(Y, \D{P}^1)\in \calh\times \calb \mid [Y]\in \D{P}^1\} $$
\end{remark}

From the proof of Proposition \ref{bundle-sc}, we learn in particular that $\D{P}(\cale)$ carries a family of instantons $\hat{\mathbf{E}}$.
Accordingly, we have the following:

\begin{corollary}\label{etale-bundle}
There exists a morphism $\psi:\D{P}(\cale)\to \call(2)$ that locally, in the étale topology, has the structure of a $\D{P}^1$-bundle.
\end{corollary}

\begin{proof}
The family $\hat{\mathbf{E}}$ induces a morphism $\psi:\D{P}(\cale)\to \calm$ that, by Propositions \ref{curve-bundle} and \ref{curves-singular}, surjects onto $\call(2)$. To prove the rest of the current proposition, we argue as in \cite[Lemma 5.3]{dima-tikho}.
We start considering an étale cover $\mathbf{W}_i\to \call(2)$
of $\call(2)$ such that each $\mathbf{W}_i\times X$ carries a universal sheaf $\mathbf{E}_i$. Define $\mathbf{G}_i:={p_1}_*(\mathbf{E}_i(1))$. This is a rank 2 vector bundle on $\mathbf{W}_i$. 
Denote by $\cale_{\mathbf{W}_i}$ the pullback of $\cale$ to $\D{P}(\cale)\times_{\call(2)}\mathbf{W}_i$, so that $\D{P}(\cale)\times_{\call(2)}\mathbf{W}_i\simeq \D{P}(\cale_{\mathbf{W}_i})$. Define $\psi_i$ as the induced morphism $\D{P}(\cale_{\mathbf{W}_i})\to \mathbf{W}_i$,
and let $\hat{\mathbf{E}}_{\mathbf{W}_i}$ denote the pullback of $\hat{\mathbf{E}}$ to $\D{P}(\cale_{\mathbf{W}_i})$; by the universal property of $\mathbf{E}_i$, $\psi_i^*(\mathbf{E}_i)\simeq \hat{\mathbf{E}}_{\mathbf{W}_i}\times \mathbf{L}$, for some line bundle $\mathbf{L}$ on $\D{P}(\cale_{\mathbf{W}_i})$.
Observe now that pulling back (\ref{univ-ex-pf}) to $\D{P}(\cale_{\mathbf{W}_i})$, we obtain an injection $\calo_{\D{P}(\cale_{\mathbf{W}_i})}(1)\otimes \mathbf{L}^{*}\hookrightarrow \psi_i^*(\mathbf{G}_i)$; this induces a morphism $\D{P}(\cale_{\mathbf{W}_i})\to \D{P}(\mathbf{G}_i)$ and once again, since this is a bijective morphism between smooth varieties, we conclude that it is an isomorphism. 
\end{proof}

From the irreducibility of $\D{P}(E)$ we deduce the irreducibility of $\call(2)$; this observation together with Lemma \ref{smooth-inst-bundle} and Proposition \ref{smooth-inst-sing} lead to the following claim.

\begin{corollary}
$\call(2)$ is a smooth and irreducible scheme of dimension 6. 
\end{corollary}

Next, we argue that $\call(2)$ is not just locally a fibration over the open subet $\calu\subset\Grp(1,F(X))$ (see Remark \ref{factor-Grass}), but that actually, it is a projective bundle isomorphic to the scheme $\calb$ defined in display \eqref{defn calb}.

\begin{proposition}\label{etale-fact}
There exists an isomorphism $\call(2)\xrightarrow{\zeta} \calb$ such that $\rho=\zeta\circ \psi$.
\end{proposition}

\begin{proof}
Set theoretically, $\zeta$ is the map sending $[E]\in \call(2)$ to the pencil of curves defined by $\D{P}(H^0(E(1)))$. Let us check that it is actually a morphism of schemes. For each open $V\subset \calb$, $\zeta^{-1}(V)=\psi(\rho^{-1}(V))$ is open in $\call(2)$ since $\psi$ is open (this is a consequence of Proposition \ref{etale-bundle}).
We have then a morphism $\calo_{\calb}(V)\to \calo_{\call(2)}(\zeta^{-1}(V))$ induced by $\calo_{\calb}(V)\to \calo_{\D{P}(\cale)}(\rho^{-1}(V))$: indeed from Corollary \ref{etale-bundle}, 
$\psi_*\calo_{\D{P}(\cale)}\simeq \calo_{\call(2)}$, thus $\calo_{\call(2)}(\zeta^{-1}(V))\simeq \calo_{\D{P}(\cale)}(\psi^{-1}(\zeta^{-1}(V)))\simeq \calo_{\D{P}(\cale)}(\rho^{-1}(V))$. 
$\zeta$ is bijective by construction thus, by the smoothness of $\call(2)$ and $\calb$, is an isomorphism.
\end{proof}

We now denote by $\overline{\cali(2)}$ the closure of $\cali(2)$ in $\calm$ and by $\overline{\cali(2)}^{inst}$ its open subscheme parameterizing instantons.

\begin{corollary}\label{closure}
$\call(2)\simeq \overline{\cali(2)}^{inst}$.
\end{corollary}

\begin{proof}
$\call(2)$ is a smooth irreducible 6-dimensional variety that contains the moduli \mbox{$\cali(2)=\call(2)\setminus \cald(1,1)$} as an open dense subset. Therefore we have equalities \mbox{$\call(2)=\overline{\call(2)}\cap \call(2)=\overline{\cali(2)}\cap \call(2):=\overline{\cali(2)}^{inst}$.}
Note that $\cali(2)$ identifies with the following open subset of $\call(2)$: for $\mathring{\calb}$, the open subset of base point free pencils, we have $\cali(2)=\zeta^{-1}(\mathring{\calb})$.
\end{proof}

\begin{corollary}
$\cald(1,1)$ is contained in $\overline{\cali(2)}$; in particular a general deformation of a non locally free instanton $[E]$ in $\call(2)$ is an instanton bundle.
\end{corollary}

\begin{proof}
$\cald(1,1)\subset \overline{\cali(2)}$ is an immediate consequence of Corollary \ref{closure}. 

The possibility to deform $[E]\in \cald(1,1)$ to an instanton bundle is due to the smoothness of $\call(2)$. Note in particular that we have the following. The locus of pencils of curves with a base point is a smooth and irreducible divisor $\calz\subset \calb=\calb\setminus \mathring{\calb}$ and it is the image of $\cald(1,1)$ trough $\zeta$.
A deformation of $E$ to an instanton bundle, for $[E]\in \cald(1,1)$ corresponds therefore to a deformation of $[\D{P}^1(H^0(E(1)))]\in \calz$ along a direction normal to $\calz$ (the smoothness of $\calb$ and $\calz$ implies that such a deformation is always possible). 
\end{proof}


\section{The moduli space $\calm_X(2;-1,2,0)$}\label{sec:moduli}

In this section we provide a full description of the moduli space $\calm$. In the previous section we proved that the closure $\overline{\call(2)}$ of the instanton moduli space is an irreducible component of $\calm$; to complete our description of $\calm$ we pass then to the study the closed subscheme $\calc:=\calm\setminus \call(2)\subset \calm$ consisting of the non instanton sheaves in $\calm$. Such sheaves can be characterized as follows.

\begin{proposition}\label{et-non-inst}
Each sheaf $E$ corresponding to a point $[E]\in \calc$ is obtained by elementary transformation of a $\mu$-stable reflexive sheaf $F$ with Chern classes $(-1,2,2)$ along a point. Conversely, for each sheaf $F$ such that $[F]\in \calm_X(-1;2,2,2)$ the kernel of a surjection $F\twoheadrightarrow \calo_p$ locates a point in $\calc$
\end{proposition}

\begin{proof}
Let us take a non instanton sheaf $E$ and consider $E^{\vee\vee}$. 
This latter must be a $\mu$-stable reflexive sheaf having $c_1(E^{\vee\vee})=-1$ and by Lemma \ref{refl-low}, either $c_2(E^{\vee\vee})=1$ and $E^{\vee\vee}\simeq \cals$ or $c_2(E^{\vee\vee})=2$.
In the first case $E^{\vee\vee}/E$ is a one-dimensional sheaf $T$ with Hilbert polynomial $n+1$; we denote by $T_0$ the maximal zero-dimensional subsheaf of $T$ and by $T_1$ the quotient $T_1:=T/T_0$.
$T_1$ is thus a line bundle on a line $l\subset X$ and since $\cals$ surjects onto $T_1$, $\cals|_{l}\simeq \calo_l(-1)\oplus\calo_l$ and $[E]\not\in \call(2)$, we conclude that $T_1\simeq \calo_l(-1)$ and that $T_0\simeq \calo_p, \ p=\supp(T_0)$. Denote by $F$ the kernel of the surection $\cals\twoheadrightarrow \calo_l(-1)$; this is a $\mu$-stable sheaf with Chern classes $(-1,2,2)$ and from the commutative diagram:
\begin{center}\label{cd-non-inst-line}
\begin{tikzcd}
0  \arrow[r] & E\arrow[r]\arrow[d] & \cals\arrow[r]\arrow[d] & T \arrow [r]\arrow[d]& 0\\
0\arrow[r] &F\arrow[r] & \cals\arrow[r]& \calo_l(-1) \arrow[r]
& 0\\
\end{tikzcd}
\end{center}
we see that $E\simeq \ker(F\twoheadrightarrow \calo_p)$. Suppose now that $c_2(E^{\vee\vee})=2$. In this case $T:=E^{\vee\vee}/E$ is a zero-dimensional and has Chern character $\ch(T)=(0,0,0,\frac{c_3(E)}{2})$.
Applying \cite[Theorem 2.2]{Ein-Sols} the spectrum of $E^{\vee\vee}$ can only consists of the integer \mbox{$k=-2=-1-\frac{c_3}{2}$.} This implies that $c_3(E)=2$ hence that $T\simeq \calo_p,\ p=\supp(T)$.  

For the converse implication we just need to check that the elementary transformation $E$ of a sheaf $F$, $[F]\in \calm_X(2;-1,2,2)$ along a point $p$ is indeed semistable.
Arguing as above, for $[F]\in \calm_X(2;-1,2,2)$ we have that $F^{\vee\vee}$ is reflexive with $c_1(F^{\vee\vee})=-1$ and $c_2(F^{\vee\vee})=1$ or $2$. In the first case $F^{\vee\vee}=\cals$; in the second, applying again \cite[Theorem 2.2]{Ein-Sols}, we get $c_3(F^{\vee\vee})=2$ hence $F\simeq F^{\vee\vee}$. In both cases $F^{\vee\vee}$ is $\mu$-stable therefore, if $E=\ker(F\twoheadrightarrow \calo_p)$, $E$ is $\mu$-stable as well since $E^{\vee\vee}\simeq F^{\vee\vee}$. 
\end{proof}

Once again we will use the Serre correspondence to deduce the geometric properties of $\calc$ from the geometry of the family of the corresponding curves; these curves will still belong to the Hilbert scheme $\Hilb_{2t+2}(X)$ but this time they wont be l.c.m.

The study of $\calc$ will lead us to prove the main result of this section:

\begin{theorem}\label{complete}
The moduli space $\calm$ is connected and consists of two irreducible components:
\begin{enumerate}
\item a 6-dimensional component $\overline{\call(2)}$ given by the closure of the open subset of instanton sheaves;
\item a 10-dimensional irreducible component $\calc$ consisting of non instanton sheaves.
\end{enumerate}
In addition, $\calm$ is generically smooth along both components. 
\end{theorem}


\subsection{The moduli space $\calm_X(2;-1,2,2)$}

In order to better understand the geometry of $\calc$, we first need to study $\calm_X(2;-1,2,2)$.
From the proof of Proposition \ref{et-non-inst} we have already learnt that we have two families of sheaves in $\calm_X(2;-1,2,2)$:
\begin{lemma}\label{class-conic}
Let $F$ be a rank 2 semistable sheaf $F$ with Chern classes $(-1,2,2)$.
Then $F$ is $\mu$-stable and either $F$ is reflexive and $\sing(F)$ is zero dimensional of length 2, or $F^{\vee\vee}\simeq \cals$ and $\cals/F\simeq \calo_l(-1)$ for a line $l\subset X$.
\end{lemma}

We are going to prove the following:
\begin{theorem}\label{sheaves-grass}
$\calm_X(2;-1,2,2)$ is a smooth 6-dimensional irreducible variety isomorphic to $\Grp(1,\D{P}(V)$.
\end{theorem}
 
Also in this occasion, our main tool to study sheaves in $\calm_X(2;-1,2,2)$ is the Serre correspondence.
Towards the rest of the section we denote by $\calr(2;-1,2,2)$ the open subset of $\calm_X(2;-1,2,2)$ parameterizing reflexive sheaves.
\begin{lemma}\label{sc-refl-conic}
Let $[F]$ be a point in $\calm_X(2;-1,2,2)$. Then $h^0(F(1))=3$ and for each $s\in H^0(E(1)), \ s\ne 0$, $\coker(s)\simeq \cali_C(1)$ for a conic $C\subset X$.
\end{lemma}
\begin{proof}
Let us start by considering the case $[F]\in \calr(2;-1,2,2)$. We show that $F(1)$ always admits global sections.
By Riemann--Roch $\chi(F(1))=3$ and by stability $h^3(F(1))=0$. As already claimed in the proof of Proposition \ref{et-non-inst}, the spectrum of $F$ consists just of the integer $-2$ (due to  \cite[Theorem 2.2]{Ein-Sols}) which implies the following:
\begin{equation}\label{spec-refl}
    h^1(F(j))=h^0(\calo_{\p1}(k+j+1)), \ \forall\: j\le -2, \ h^2(F(j))=h^1(\calo_{\p1}(k+j+1)), \ \forall\: j\ge 0.
\end{equation}
This means that $h^2(F(1))=0$ hence we necessarily have $h^0(F(1))>0$. 
Now, since $H^0(F)=0$, $\forall \ s\in H^0(F(1)),\: s\ne 0$ $\coker(s)$ is torsion free 
and of the form $\cali_C(1)$ for a l.c.m. curve $C\subset X$.

A Chern class computation leads to $P_C(n)=2n+1$ hence $C$ is a plane conic and from the short exact sequence
\begin{equation}\label{ses-refl-conic}
0\rightarrow \calo_X\rightarrow F(1)\rightarrow \cali_C(1)\rightarrow 0
\end{equation}
we compute that $h^0(F(1))=3$ hence $h^1(F(1))=0$.
If $F^{\vee\vee}\simeq \cals$ instead, from the short exact sequence
$$ 0\rightarrow F\rightarrow \cals \rightarrow\calo_l(-1)\rightarrow 0,$$
we get $h^i(F(1))=h^i(\cals(1))=0, \ i=2,3$; as moreover $H^0(\cals(1))\rightarrow H^0(\calo_l)$ can not be the zero map (for $l'\subset X$ general, there are no surjection $\cali_{l'}(1)\twoheadrightarrow \calo_l)$ we conclude that $h^0(F(1))=3$ and $h^1(F(1))=0$. Since any global section of $\cals$ has torsion free cokernel, for any  non-zero $s\in H^0(F(1))$, $\coker(s)\simeq \cali_Z(1)$ fits in
\begin{equation}\label{ses-reducible}
0\rightarrow \cali_Z(1)\rightarrow \cali_{l'}(1)\rightarrow \calo_l\rightarrow 0
.\end{equation}
with $l'=\coker(\iota(s)), \ \iota:=H^0(F(1))\hookrightarrow H^0(\cals(1)).$ Therefore $Z$ is a reducible conic supported on $l\cup l'$ (note that for $s$ general, as $\cali_{l'}$ surjects onto $\calo_l(-1)$, $l'$ will meet $l$ at a point).
\end{proof}

It is straightforward to check that this construction can be ``reversed", leading to the following claim.

\begin{lemma}
Serre correspondence establishes a 1-1 correspondence between
\begin{itemize}
    \item pairs $(F,s)$ with $[F]\in \calr(2;-1,2,2),\ s\in \D{P}(H^0(F(1)))$ (resp. pairs $(F,s)$ with $[F]\in \calm_X(2;-1,2,2)\setminus\calr(2;-1,2,2),\ s\in \D{P}(H^0(F(1)))$)
    \item  pairs \mbox{$(C,\xi)$ with $[C]\in \Hilb_{2t+1}(X),$}
$\ \xi\in \D{P}(H^0(\omega_C(2)))$ vanishing along 2 points on $C$ (resp.
$(C,\xi), \ [C]\in \Hilb_{2t+1}(X),$ reducible 
$\ \xi\in \D{P}(H^0(\omega_C(2)))$ vanishing along a component of $C$)
\end{itemize}
\end{lemma}

Mimicking what we have done for sheaves in $\call(2)$, we are now going to describe in details the linear system $\D{P}(H^0(F(1)))\simeq \p2$ of conics associated to $[F]\in \calm_X(-2;-1,2,2)$. This will help us to better understand the geometry of the scheme $\calm_X(2;-1,2,2)$.

\begin{proposition}\label{ls-conic}
\begin{enumerate}
    \item If $[F]\in \calr(2;-1,2,2)$ and $\sing(F)$ consists of two distinct points $p_1,\: p_2$, $\D{P}(H^0(F(1)))$ identifies with the linear system of conics containing $p_i, \ i=1,2$ and its image under the isomorphism \mbox{$\Hilb_{2t+1}(X)\simeq \Grp(2,V^*)$} is the Schubert variety of planes containing $\overline{p_1p_2}.$
    \item If $[F]\in \calr(2;-1,2,2)$ and $\sing(F)_{\rm red}=p$, there exists a line $l\subset \D{P}^4$ tangent at $p$ to each conic in $\D{P}(H^0(F(1)))$ and the image of $\D{P}(H^0(F(1)))$ under the isomorphism $\Hilb_{2t+1}(X)\simeq \Grp(2,V^{*})$ is the Schubert variety of planes containing $l$.
    \item If $[F]\in \calm_X(2;-1,2,2)\setminus \calr(2;-1,2,2)$ $\D{P}(H^0(F(1)))$ identifies with the Schubert variety of planes containing the line $\sing(F)$. 
\end{enumerate}
\end{proposition}
\begin{proof}
Suppose at first that $F$ reflexive with $\sing(F)=\{p_1,\:p_2\},\ p_1\ne p_2$. It is easy then to compute that the conics in $X$ passing through these points are parameterised by a plane: a conic $C\subset X$ passes indeed through the points $p_1,\ p_2$ if and only if $\overline{p_1p_2}\subset \langle C\rangle$. This means that the image of the family of conics passing through the $p_i$s, under the isomorphism $\Hilb_{2t+1}(X)\simeq \Grp(2,V^*), \ C\mapsto \langle C\rangle $ is the Schubert variety of planes containing $\overline{p_1p_2}$ that is a plane in $\Grp(2,V^*)$. The proposition follows since every conic in $\D{P}(H^0(F(1)))$ contains $p_i, \ i=1,2$.
Whenever $\sing(F)$ is supported on a single point $p$ we can compute again that the linear system of conics containing $\sing(F)$ is a plane in $\Grp(2,V^*)$. The scheme $\sing(F)$ corresponds indeed to the data of the point $p$ together with a tangent direction $v\in T_p X$ or equivalently, to a line $l$ tangent to $X$ at $p$. A conic $C$ contains $\sing(F)$ if and only if $\langle C\rangle $ contains $l$, hence $\D{P}(H^0(F(1)))$ identifies with the Schubert variety $\p2\subset \Grp(2,V^*)$ of planes containing $l$.  
Finally if $F$ is singular along a line $l$, consider the inclusion $\iota:H^0(F(1))\hookrightarrow H^0(\cals(1))$. Each $\iota(s)$ defines a line $l'$ giving rise to a short exact sequence of the form (\ref{ses-reducible}). 
Since $\cali_{l'}$ surjects onto $\calo_l(-1)$ if and only if either $l=l'$ or $l\cap l'$ consists of a point, we deduce therefore that $\D{P}(\iota(H^0(F(1)))$ identifies with the space of lines meeting $l$ and that $\D{P}(H^0(F(1))$ identifies therefore with the family of planes containing $l$.
\end{proof}

From now on the family of conics associated to $[F]\in \calm_X(2;-1,2,2)$ will simply be denoted by $\D{P}(H^0(F(1)))$ and the line contained in every plane $\langle C\rangle$, \mbox{$C\in \D{P}(H^0(F(1)))$} will be denoted by $l_{F}$ (note that for $[F]$ belonging to the closed subscheme $\calm_X(2;-1,2,2)\setminus \calr(2;-1,2,2)$, $l_F=\sing(F)$).
\begin{lemma}\label{lin-refl}
$[F]\in \calr(2;-1,2,2)$ if and only if $l_F\not \subset X$.
\end{lemma}
\begin{proof}
If $l_F\subset X$, all the conics $C\in \D{P}(H^0(F(1)))$ contain $l_F$ and \mbox{$\sing(F)\subset l_F$.} This can not happen if $F$ is reflexive, since if ever a section $\xi\in H^0(\omega_C(2))$, \mbox{$ [C]\in \D{P}(H^0(F(1))$} vanishes along 2 points on $l_F\subset C$, it would vanish along the entire $l_F$ contradicting the reflexivity of $F$.
The converse implication is obvious since $l_F\not\subset X$ ensures that $\sing(F)=l_F\cap X$ consists of two points hence, by Lemma \ref{class-conic}, $[F]\in \calr(2;-1,2,2)$.
\end{proof}

Let us now analyse the local behaviour of the moduli space $\calm_X(2;-1,2,2)$.
\begin{proposition}\label{lines-smooth}
For each point $[F]\in \calm_X(2;-1,2,2)$ we have $\ext^2(F,F)=0$.
For each point $[F]\in \calr(2;-1,2,2)$ we have $\ext^1(F,F)=6$.
\end{proposition}

\begin{proof}
A general section $s\in H^0(F(1))$ defines a short exact sequence of the form (\ref{ses-refl-conic}). 

Applying the functor $\Hom(\: \cdot\:, F)$ we get a sequence:
    $$\Ext^2(\mathcal{I}_C, F)\rightarrow \Ext^2(F,F)\rightarrow \Ext^2(\OO_X(-1),F);$$
The r.s.t is zero since, from (\ref{spec-refl}), $H^2(F(1))=0$; from (\ref{ses-refl-conic}) we compute the vanishing of $H^i(F)=0, \ i=2,3$ yielding:
$$ \Ext^2(\mathcal{I}_C, F)\simeq \Ext^3(\OO_C, F)\simeq \Hom(F, \OO_C(-3))^*.$$

Applying $-\otimes \OO_C$ to (\ref{ses-refl-conic}), we obtain an exact sequence
$$0\rightarrow \intor_1^{\OO_X}(F,\mathcal{O}_C)\xrightarrow{a} 
\intor_1^{\OO_X}(\mathcal{I}_C, \OO_C)\xrightarrow{b} \OO_C(-1)\xrightarrow{c} F|_C\xrightarrow{d} \mathcal{N}_{C/X}^{\vee}\rightarrow 0$$

from which we extract the short exact sequences:
\begin{equation}\label{delta}
0\rightarrow \ker(d)\rightarrow F|_C\xrightarrow{d} \mathcal{N}_{C/X}^{\vee}\rightarrow 0
\end{equation}

\begin{equation}\label{gamma}
    0\rightarrow \ker(c)\rightarrow \OO_C(-1)\rightarrow \ker(d)\rightarrow 0
\end{equation}
Applying $\Hom(\ \cdot\ , \OO_C(-3))$ to (\ref{gamma}), we get that $\Hom(\ker(c), \OO_C(-3))$ injects into $\Hom(\OO_C(-1),\OO_C(-3))\simeq H^0(\OO_C(-2))=0$. Therefore, from (\ref{delta}), we get $\Hom(F|_C, \OO_C(-6pt))\simeq \Hom(\mathcal{N}_{C/X}^{\vee}, \OO_C(-6pt))\simeq H^0(\mathcal{N}_{C/X}(-3)).$
$C$ is a plane section of $X$, therefore $\caln_{C/X}\simeq \calo_C(1)^{\oplus 2}$ hence $H^0(\mathcal{N}_{C/X}(-3))=0$ implying $\Hom(F, \OO_C(-3))=0$ and finally, $\Ext^2(F,F)=0$.
If we suppose that moreover $F$ is reflexive, it has homological dimension 1; this allows to compute (again arguing as in \cite[Proposition 3.4]{H-reflexive})
$$\chi(F,F)= -5.$$
The stability of $F$ implies that $\hom(F,F)=1$ and that $\ext^3(F,F)=0$; from our previous arguments $\ext^2(F,F)=0$ hence
$\ext^1(F,F)=6$. 
\end{proof}

We consider now $\Hilb_{t^2+tn+2}(\Grp(1,\D{P}(V^*)))$, the Hilbert scheme of planes in $\Grp(1,\D{P}(V^*))$. Recall that this scheme has two components:
a component $\Omega$ parameterizing families of planes contained in a same hyperplane and a second component $\Lambda$ parameterizing families of planes $\Lambda_l$ containing a fixed line $l$. This latter is isomorphic to $\Grp(1,\D{P}(V))$ via the morphism: 
$$\Grp(1,\D{P}(V))\longrightarrow \Lambda, l\mapsto \Lambda_l\simeq \Grp(1,\D{P}(H^0(\cali_l(1))).$$

We consider now the map: 
$$\calm_X(2;-1,2,2)\xrightarrow{\alpha}\Lambda\simeq \Grp(1,\D{P}(V)),\ [F]\mapsto \D{P}(H^0(F(1)))\leftrightarrow l_F.$$
\begin{proposition}
$\alpha$ is an isomorphism of scheme; it identifies $\calr(2;-1,2,2)$ (resp. $\calm_X(2;-1,2,2)\setminus \calr(2;-1,2,2)$) with $\Grp(1,\D{P}(V))\setminus F(X)$ (resp. $F(X)$).
\end{proposition}

\begin{proof}
We apply verbatim the arguments used in the proof of Proposition \ref{bundle-sc}.
Doing so we show that the sheaf  $\calg:=\inext^1_{p_1}(\cali_{\mathbf{C}},\calo_{X\times \Gr}(-1))$ on $\Hilb_{2t+1}\simeq \Grp(1,\D{P}(V)^*)$,  for $\mathbf{C}\subset \Grp(1,\D{P}(V^*))\times X$ the universal curve, is locally free of rank 3 and that the projective bundle $\D{P}(\calg)$ carries a family $\hat{\mathbf{F}}\in Coh(\D{P}(\calg)\times X)$ such that $[\hat{\mathbf{F}}_{(C,e)}]$ is the sheaf constructed from $e\in \Ext^1(\cali_C,\calo_X(-1))$. ${p_1}_*(\hat{\mathbf{F}}(1))$ defines a family of linear systems of conics over $\D{P}(\calg)$ inducing a morphism $\D{P}(\calg)\xrightarrow{\gamma} \Lambda$ such that ${p_1}_*(\hat{\mathbf{F}}(1))\simeq \gamma^*(\calt)$ for $\calt$ the tautological rank 3 bundle over $\Lambda$ (that is to say, $\calt$ is the bundle whose fibre over $\Lambda_l\in \Lambda$ is the vector space $\bigwedge^2 (H^0(\cali_l(1)))\simeq \D{C}^3$ of planes belonging to $\Lambda_l$).
$\gamma$ is the morphism mapping $(C,e)$ to $\Lambda_{l_F}\simeq\D{P}(H^0(F(1)))$ for $F$ the sheaf arising from $e\in Ext^1(\cali_C,\calo_X(-1)).$

The family $\hat{\mathbf{F}}$ induces a morphism \mbox{$\D{P}(\calg)\xrightarrow{\beta} \calm_X(2;-1,2,2)$}, and applying the argument used in the proof of Corollary \ref{etale-bundle} we show that $\beta$ is, in the étale topology, a $\D{P}^2$-bundle. In this way we also deduce that $\calm_X(2;-1,2,2)$ is irreducible of dimension 6 hence, by Proposition \ref{lines-smooth}, we get $\ext^1(F,F)=6, \ \forall\:[F]\in \calm_X(2;-1,2,2).$
Reasoning then as in Proposition \ref{etale-fact}, we show that, due to the properties of $\beta$, $\alpha$ is well defined as a morphism of scheme.
Since $\alpha$ maps bijectively $\calm_X(2;-1,2,2)$ into $\Lambda\simeq \Grp(1,\D{P}(V))$ and since both schemes are smooth, we conclude that $\alpha$ is an isomorphism.  The fact that $\alpha(\calr(2;-1,2,2))=\Grp(1,\D{P}(V))\setminus F(X)$ is due to lemma \ref{lin-refl}. This ends the proof of the proposition.
\end{proof}

This completes the proof of Theorem \ref{sheaves-grass}.

\subsection{Description of $\calc$}

We can finally come back to the description of $\calc$. 

\begin{proposition}\label{sc-non-inst}
For $[E]\in \calc$, $h^0(E(1))=2$ and for all $s\in H^0(E(1)), \ s\ne 0$, $\coker(s)\simeq\cali_{\Gamma}(1)$, for $\Gamma$ a curve union of a conic and a point. More precisely, all the curves $\Gamma$ in $\D{P}(H^0(E(1)))$ are of the form:
$$ 0\rightarrow \calo_p\rightarrow \calo_{\Gamma}\rightarrow \calo_C\rightarrow 0$$
with $p$ fixed and with $C$ varying in a pencil of conics contained in a fixed hyperplane.
\end{proposition}

\begin{proof}
From Proposition \ref{et-non-inst}, $E$ always fits in a short exact sequence:
\begin{equation}\label{ses-non-inst}
0\rightarrow E\rightarrow F\rightarrow \calo_p\rightarrow 0
\end{equation}
with $[F]\in \calm_X(2;-1,2,2)$. 
Twisting (\ref{ses-non-inst}) and taking global section, we deduce that $h^0(E(1))\ne 0$; moreover the fact that $\forall s\in H^0(F(1)),  \coker(s)$ is torsion free, ensures that the same holds for all non-zero $s\in H^0(E(1))$. As usual we denote by $\iota$ the inclusion $\iota:H^0(E(1))\hookrightarrow H^0(F(1))$. For any non-zero $s\in H^0(E(1))$ we therefore have 

\begin{center}\label{cd-non-inst}
\begin{tikzcd}
0  \arrow[r] &\OO_X\arrow[r]\arrow[d, "s"] &\OO_X\arrow[d, "\iota(s)"]\arrow[r] & 0 \arrow[d] \arrow [r] &0 \\
0  \arrow[r] & E(1)\arrow[r]\arrow[d, two heads] & F(1)\arrow[r]\arrow[d, two heads] & \OO_{p}(1) \arrow [r]\arrow[d, "id"]& 0\\
0\arrow[r] &\cali_{\Gamma}(1)\arrow[r] & \cali_{C}(1)\arrow[r]& \calo_p(1) \arrow[r]
& 0\\
\end{tikzcd}
\end{center}
from which we deduce that $\coker(s)\simeq \cali_{\Gamma}(1)$, for $\Gamma$ a curve with Hilbert polynomial $2t+2$ and supported on $C\cup p$. Since we can have no plane section of $X$ containing $\Gamma$ we have $h^0(\cali_{\Gamma}(1))=1$ hence $h^0(E(1))=2$.
Now, $\D{P}(\iota(H^0(E(1)))$ is a pencil in $\D{P}(H^0(F(1)))\simeq \Lambda_{l_F}$, therefore there exists a unique hyperplane section containing all the conics in $\D{P}(\iota(H^0(E(1)))$. 
\end{proof}

\begin{remark}
Suppose that $p\not\in \sing(F)$. From the proof of the previous proposition we learn that each $f\in \D{P}(\Hom(F,\calo_p))$ locates a unique hyperplane $H\in \D{P}(H^0(\cali_{\langle l_F,p\rangle}(1)))$ containing the curves $\D{P}(H^0(E(1)))$,  $E:=\ker(F)$. For $C\in \D{P}(\iota(H^0(E(1))))$ general, this is the hyperplane generated by $\langle C\rangle $ and $p$.
Notice also that for $f, \ f'\in \D{P}(\Hom(F,\calo_p)), \ f\ne f'$, denoting by $E:=\ker(f),\ E':=\ker(f')$, the pencils $\D{P}(\iota(H^0(E(1))))$ and $\D{P}(\iota(H^0(E'(1))))$ meet just at $\langle l_F,p\rangle \cap X$.  
If ever $\D{P}(\iota(H^0(E(1))))= \D{P}(\iota(H^0(E'(1))))$, we would indeed have that for all conic $C\in \D{P}(\iota(H^0(E(1)))$ such that $p\not\in C$, $C\cup p\in \D{P}(H^0(E(1)))\cap \D{P}(H^0(E'(1)))$. Since $\Ext^1(\cali_{C\cup p},\calo_X(-1))\simeq \Ext^1(\cali_C,\calo_X(-1))$, $E$ and $E'$ would then both arise from the unique extension class image of the element $\xi \in \Ext^1(\cali_C,\calo_X(-1))$ defining $F$, which would lead to $E\simeq E'$, a contradiction. 

We describe in this way a pencil (parameterized by $\D{P}(\Hom(F,\calo_p))$) of lines in $\D{P}(H^0(F(1)))$ that identifies with the family of hyperplanes containing $\langle l_F, p\rangle$.
\end{remark}

Now, since each point $[E]\in \calc$ uniquely determine a pair $(F,p)\in \calm_X(2;-1,2,2)$, we have a well defined map (for the moment just defined at level of sets) 
$$\delta: \calc\longrightarrow \calm_X(2;-1,2,2), \ [E]\mapsto (F,p)$$
where $(F,p)$ are such that $E$ is obtained by elementary transformation of $F$ along $p$. Consider now the open subset $(\calm_X(2;-1,2,2)\times X)_0$ parameterizing pairs $(F,p)$ such that $p\not\in \sing(F)$ and denote by ${\calc}_0\subset \calc$ its preimage under $\delta$.

\begin{proposition}\label{tangent-C}
For each point $[E]\in \calc_0, \ \ext^1(E,E)=10$.
\end{proposition}

\begin{proof}
We know that $E$ always fits in a short exact sequence: 
\begin{equation}\label{ses-dd-point}
    0\rightarrow E\rightarrow F \rightarrow \OO_p\rightarrow 0
\end{equation}
with $[F]\in \calm_X(2;-1,2,2)$. We apply $\Hom(\:\cdot\: , E)$ to it.
We can see immediately that the stability of $E$ and of $F$ imposes $\Hom(F,E)=0$ 
therefore we get:
\begin{align}\begin{split}
0\rightarrow \Hom(E,E)\rightarrow \Ext^1(\OO_p,E)\rightarrow \Ext^1(F,E)\rightarrow\\ 
\rightarrow \Ext^1(E,E)\rightarrow \Ext^2(\OO_p, E)\rightarrow \Ext^2(F,E)
\end{split}\label{lunga-ext}\end{align}
The term $\Ext^2(F,E)$ fits into:
$$\Ext^1(F,\OO_p)\rightarrow \Ext^2(F,E)\rightarrow \Ext^2(F,F)$$
The r.s.t vanishes due to Proposition \ref{lines-smooth}; since moreover we are assuming $[E]\in \calc_0$, $F$ is locally free at $p$ so that $\Ext^1(F,\calo_p)\simeq H^1(\inhom(F,\calo_p))=0$.

These computations lead to $\Ext^2(F,E)=0$.
Let us now compute the dimensions of the spaces $\Ext^{3-i}(\OO_p,E)\simeq \Ext^{i}(E,\OO_p(-3))^*, \ i=1,2$.
Again, since $p\not\in\sing(F), \ \Ext^i(F,\calo_p(-3))=0, \ i=1,2$,
therefore 
$$\Ext^i(E,\calo_p(-3))\simeq \Ext^{i+1}(\calo_p,\calo_p(-3))\simeq \Ext^{2-i}(\calo_p,\calo_p)^*, \ i=1,2.$$
For $i=2$ we thus get $\ext^1(\calo_p,E)=1=\hom(\calo_p,\calo_p)$ whilst for $i=1$ we obtain $\ext^2(\calo_p,E)=\ext^1(\calo_p\calo_p)=h^0(\caln_{p/X})=3$.
From \eqref{lunga-ext} we therefore compute that $\ext^1(E,E)=10$.
\end{proof}

\begin{proposition}\label{C0-sm}
${\calc}_0$ is a smooth 10-dimensional irreducible scheme.
\end{proposition}
\begin{proof}
We will construct a $\p1$ bundle $\D{P}(\cala)$ over $(\calm_X(2;-1,2,2) \times X)_0$
and show that this is endowed with a morphism $\D{P}(\cala)\to \calm$ mapping $\D{P}(\cala)$ bijectively into $\calc_0$.
We consider the Grassmanniann of lines $\Grp(1,\D{P}(V))$ in $\D{P}(V)\simeq \p4$. For the ease of notations towards the rest of the proof this latter will always be denoted simply by $\Gr$.
We define $(\Gr\times X)_0$ as the open set:
$$({\Gr\times X})_0:=\{([l],p)\in \Gr \times X\mid p\not\in l\}. $$
This scheme carries a family of planes $\tilde{\Pi}\subset {(\Gr\times X)}_0\times X$ such that \mbox{$\tilde{\Pi}_{(l,p)}\simeq \langle l,p\rangle$} 
and the sheaf $\cala:={p_1}_*(\cali_{\tilde{\Pi}}(1))$ is a rank 2 vector bundle on $({\Gr\times X})_0$. 

Now, the isomorphism $\beta:\calm_X(2;-1,2,2)\to\Gr$ (see Proposition \ref{sheaves-grass}) induces an isomorphism
${(\Gr\times X)}_0\xrightarrow{\simeq}{(\calm_X(2;-1,2,2)\times X)}_0$, hence $\D{P}(\cala)$ is a $\p1$-bundle over ${(\calm_X(2;-1,2,2)\times X)}_0$ as well (the fibre over a point $(F,p)$ identifies with the pencil $\D{P}(H^0(\cali_{\langle l_F,p\rangle}(1)))$ of hyperplanes containing $\langle l_F,p\rangle$).
In order to prove that $\D{P}(\cala)$ admits a morphism to $\calm$, we consider an étale cover $\mathbf{W}$ of $\calm_X(2;-1,2,2)$ supporting a universal sheaf $\bff_{\bfw}$. This induces an étale cover $\tilde{\bfw}\to {(\calm_X(2;-1,2,2)\times X)}_0$, we denote by $\tilde{\cala}_{\bfw}$ the pullback of $\cala$ to $\tilde{\bfw}$ and by $\tilde{\bff}_{\bfw}$ the pullback of $\bff_{\bfw}$ to $\tilde{\bfw}\times X$. $\D{P}(\tilde{\cala}_{\bfw})$ identifies with the $\p1$ subbundle of $G_2({p_1}_*(\tilde{\bff}_{\bfw}(1)))$ 
whose fibre over a point $(w,p)\in \tilde{\bfw}$ is the 1-dimensional linear space:
$$\pi_{\tilde{\cala}_{{\bfw}}}^{-1}(w,p)=\{\D{P}^1\subset \D{P}(H^0(\bff_{w}(1)))\simeq \Lambda_{l_{\bff_{w}}}\mid [\langle l_{\bff_{w}}, p\rangle]\in \D{P}^1\} $$

Define now $\tilde{\Delta}_{\bfw}$ as the pullback to $\tilde{\bfw}\times X$ of the diagonal $\Delta\subset X\times X$
and consider the sheaf $\tau_{\bfw}:={p_1}_*(\inhom(\tilde{\bff}_{\bfw},\calo_{\tilde{\Delta}_{\bfw}}))$.
This is a rank 2 vector bundle over $\tilde{\bfw}\times X$ whose fibre over $(w,p)$ identifies with $\Hom(\bff_{w},\calo_p).$
We claim that we have an isomorphism: $\D{P}(\tilde{\cala}_{{\bfw}})\simeq \D{P}(\tau_{\bfw}).$
Denote by $\hat{\bff}_{\bfw}, \hat{\Delta}_{\bfw}$ the pullback to $\D{P}(\tau_{\bfw})$ of $\tilde{\bff}_{\bfw}, \ \tilde{\Delta}_{\bfw}$, respectively. 
The image of the identity $id_{\tau_{\bfw}}\in \End(\tau_{\bfw})$ through the isomorphism:
\begin{align*}
    &\Hom(\tau_{\bfw},\tau_{\bfw})\simeq H^0(\tau_{\bfw}\otimes \tau_{\bfw}^{\vee})\simeq H^0(\tau_{\bfw}\otimes {\pi_{\tau_{\bfw}}}_*\calo(1)) \simeq \\
    \simeq &~ H^0({p_1}_*\inhom(\hat{\bff}_{\bfw},\calo_{\hat{\Delta}_{\bfw}})\otimes \calo_{\D{P}(\tau_{\bfw})}(1))\simeq \Hom(\hat{\bff}_{\bfw}, \calo_{\hat{\Delta}_{\bfw}}\otimes {p_1}^*\calo_{\D{P}(\tau_{\bfw})}(1))
\end{align*}
defines a morphism $\hat{\bff}_{\bfw}\to \calo_{\hat{\Delta}_{\bfw}}\otimes {p_1}^*\calo(1)$
inducing a short exact sequence on  $\D{P}(\tau_{\bfw})\times X$:
$$ 0\rightarrow \hat{\bfe}_{\bfw}\rightarrow \hat{\bff}_{\bfw}\rightarrow  \calo_{\hat{\Delta}_{\bfw}}\otimes {p_1}^*\calo_{\D{P}(\tau_{\bfw})}(1)\rightarrow 0.$$
Twisting and applying ${p_1}_*$, we obtain a rank 2 vector bundle ${p_1}_*(\hat{\bfe}_{\bfw}(1))$; this latter induces an embedding $\D{P}(\tau_{\bfw})\to G_2({p_1}_*\tilde{\bff}_{\bfw}(1))$
that maps $\D{P}(\tau_{\bfw})$ bijectively into $\D{P}(\tilde{\cala}_{{\bfw}})$. This induces an isomorphism $\D{P}(\tau_{\bfw})\simeq \D{P}(\tilde{\cala}_{\bfw})$ mapping a point $f\in \Hom(\bff_w,\calo_p)$ in $\pi_{\tau_{\bfw}}^{-1}$ to $\D{P}(H^0(E_w(1))), \ E_w=\ker f$.
The sheaves $\hat{\bfe}_{{\bfw}}$ determine morphisms $\psi_{\tilde{\bfw}}:\D{P}(\tilde{\cala}_{\bfw})\xrightarrow{\simeq}\D{P}(\tau_{\bfw})\to \calm$ that descend to a morphism $\D{P}(\cala)\to \calm$ which maps $\D{P}(\cala)$ bijectively to $\calc_0$.
This means that $\calc_0$ is irreducible and has dimension $10$. Since by Proposition \ref{tangent-C}, the dimension of the Zarisky tangent space at each point $[E]\in \calc_0$ is 10, we conclude that $\calc_0$ inherits from $\D{P}(\cala)$ the structure of a smooth 10-dimensional scheme.
\end{proof}

From these last results we deduce that $\overline{\calc_0}$ is an irreducible component of $\calc$ and that this latter is smooth along $\calc_0$. We finally want to show that actually, we have an equality $\calc=\overline{\calc_0}$, that is to say, that $\calc$ is irreducible.

\begin{proposition}
$\calc$ is an irreducible 10-dimensional scheme that coincides with $\overline{\calc_0}$.
\end{proposition}

\begin{proof}
We show that $\calc_0$ is dense in $\calc$.
From Proposition \ref{sc-non-inst}, we learn that pairs $(E,s), \ [E]\in \calc\setminus \calc_0, \ s\in \D{P}(H^0(E(1))$ corresponds to pairs $(\Gamma,\xi)$ with $\Gamma$ a non l.cm. curve consisting of a conic $C$ with an embedded point $p\in C$ and $\xi\in \Ext^1(\cali_{\Gamma},\calo_X(-1))\simeq \Ext^1(\cali_{C},\calo_X(-1))\simeq H^0(\calo_C(2pt))$ vanishing at $p$.
For a point $[E]\in \calc\setminus \calc_0$ let then $\Gamma$ be a curve defined by a global section $s\in H^0(E(1))$ and $\xi$ the corresponding element in $\Ext^1(\cali_{\Gamma},\calo_X(-1))$.
The projective space $\D{P}(\Ext^1(\cali_{\Gamma},\calo_X(-1)))\simeq \D{P}^2$ determines a family $\mathbf{E}$, flat over $\D{P}^2$ such that $\mathbf{E}_{[\xi]}\simeq E$. Now, the points $x\in \D{P}(\Ext^1(\cali_{\Gamma},\calo_X(-1)))$ such that $\mathbf{E}_x\in \calc\setminus \calc_0$ are parameterized by a divisor isomorphic to $\D{P}^1$. We can therefore always deform $[\xi]$ in a direction normal to this divisor and produce in this way a family $\mathbf{E}'$ whose central fibre is isomorphic to $[E]$ and whose general fibre lies in $\calc_0$. 
\end{proof}

\subsection{Intersection of $\calc$ and $\overline{\call(2)}$}
We end our description of the moduli space $\calm$ addressing the issue of its connectedness. Since we have already proved that $\calm$ is union of two irreducible components, this reduces to proving that their intersection is non-empty. 
We present here how to construct a 5-dimensional irreducible family $\calp$ contained in 
$\calc\cap\overline{\call(2)}$.
To begin with we consider $\Tan(X)\subset \Grp(1,\D{P}(V))$, the variety of tangent lines to $X$.
This is a smooth 4-dimensional locally closed subvariety of $\Grp(1,\D{P}(V))$ and it identifies with the subvariety $\Sigma_{T}\subset \Grp(1,\D{P}(V))\times X$ defined as
$$\Sigma_{T}:=\{([l],p)\in \Tan(X)\times X\mid l\subset \D{T}_{p}X\}.$$
We denote by $\calr_T\subset \calr(2;-1,2,2)\times X$ the image of $\Sigma_T$ under the isomorphism $\Grp(1,\D{P}(V))\times X\xrightarrow{\simeq} \calm_X(2;-1,2,2)\times X$.
By definition, for each $([F],p)\in \calr_T$, $l_F$ is tangent to $X$ at $p=\sing(F)_{\rm red}$.
Now, reasoning exactly as in the proof of Proposition \ref{C0-sm}, starting from an étale cover $\bfw$ of $\calm_X(2;-1,2,2)$ supporting a Poincare sheaf $\bff_{\bfw}$, we construct an étale cover $\bfw_T$ of $\calr_T$ and we consider the sheaf ${p_1}_*(\inhom(\bff_{\bfw_T},\calo_{\Delta_{\bfw_T}}))$ on $\bfw_T$, where $\bff_{\bfw_T}\in Coh(\bfw_T\times X)$ and ${\Delta_{\bfw_T}}\subset \bfw_T\times X$ are, respectively, the pullback of the universal sheaf and of the diagonal.

The sheaf ${p_1}_*(\inhom(\bff_{\bfw_T},\calo_{\Delta_{\bfw_T}}))$ is a rank 3 vector bundle 
and replying the reasoning presented in the proof of Proposition \ref{C0-sm}, on $\D{P}({p_1}_*(\inhom(\bff_{\bfw_T},\calo_{\Delta_{\bfw_T}}))\times X$, there exists a family $\hat{\bfe}_{\bfw_T}\hookrightarrow \hat{\bff}_{\bfw_T}$, $\hat{\bff}_{\bfw_T}$ being the pullback of $\bff_{\bfw_T}$,
whose direct image under the projection $p_1$ fits in:
$$ 0\rightarrow {p_1}_*(\hat{\bfe}_{\bfw_T}(1))\rightarrow {p_1}_*(\hat{\bff}_{\bfw_T}(1))\rightarrow {p_1}_*(\calo_{\hat{\Delta}_{\bfw_T}})\otimes \calo(1)\rightarrow 0.$$
Denote now by $\bfl_{\bfw_T}\subset \Sigma_T\times_{\calr_T} \bfw_T$ the pullback of the universal line \mbox{$\bfl\subset \Sigma_T\times X$.} We observe that the projective bundle  $\D{P}({p_1}_*\cali_{\bfw_T}(1))$ identifies with the Grassmann bundle $G_2({p_1}_*({\bff}_{\bfw_T}(1))$ (since for $F\in \calr(2;-1,2,2)$ such that $p\in l_F$ $\Grp(1, \D{P}(H^0(F(1))))\simeq \D{P}(H^0(\cali_{l_F}(1)))\simeq \D{P}^2$); as  
${p_1}_*(\hat{\bfe}_{\bfw_T}(1))$ is a rank-2 subbundle of ${p_1}_*(\hat{\bff}_{\bfw_T}(1))$ it induces therefore a morphism of $\Sigma_T$ schemes:
$$ \epsilon:\D{P}({p_1}_*(\inhom(\bff_{\bfw_T},\calo_{\Delta_{\bfw_T}}))\rightarrow \D{P}({p_1}_*\cali_{\bfl_{\bfw_T}}(1))$$
($\epsilon$ is the morphism mapping $f\in Hom(F,\calo_p)$ to the unique hyperplane containing all curves in $\D{P}(H^0(E(1)))$ for $E:=\ker f$).
This morphism is bijective hence, by the smoothness of both $\D{P}({p_1}_*(\inhom(\bff_{\bfw_T},\calo_{\Delta_{\bfw_T}}))$ and $\D{P}({p_1}_*\cali_{\bfl_{\bfw_T}}(1))$ it is an isomorphism. 

Also this time, the families of sheaves $\hat{\bfe}_{\bfw_T}$ induces morphisms $\D{P}({p_1}_*\cali_{\bfl_{\bfw_T}}(1))\to \calc$ that descend to a well defined morphism $\psi:\D{P}({p_1}_*\cali_{\bfl}(1))\to \calc$.
Consider now the variety
$$\tilde{\Sigma}_T:=\{([l],p,H)\in \Sigma_T\times X^*\mid H=\D{T}_pX\}$$
$\tilde{\Sigma}_T$ is isomorphic to $\Sigma_T$ and it identifies with a subset of $\D{P}({p_1}_*(\cali_{\bfl}(1))$; we finally define $\calp$ as the scheme theoretic image of $\tilde{\Sigma}_T$ under the morphism $\psi$.

\begin{proposition}
$\calp$ is a 5 dimensional irreducible scheme contained in $\overline{\call(2)}\cap \calc$.
\end{proposition}

\begin{proof}
The dimension and the irreducibility of $\calp$ are immediate consequences of the fact that $\tilde{\Sigma}_T$ is irreducible and of dimension 5.
Let us now prove that $\calp$ is indeed contained in both components of $\calm$.
For a general point $[E]\in \psi(\tilde{\Sigma}_T)$, $E^{\vee\vee}:=F\in \calr(2;-1,2,2)$ and $F/E\simeq \calo_p, \ \sing(F)_{\rm red}=p$.
Let us give a geometric interpretation of the family of curves defined by $H^0(E(1))$. Consider the short exact sequence:
$$0\rightarrow H^0(E(1))\xrightarrow{\iota} H^0(F(1))\rightarrow H^0(\calo_p)\rightarrow 0.$$
By definition the pencil $\D{P}(\iota(H^0(E(1))))$ identifies with the pencil of planes containing $l_F$ and contained in $\D{T}_pX$; accordingly it identifies with the locus of singular conics in $\D{P}(H^0(F(1))$
(to see this just notice that the singular plane sections in $\Lambda_{l_F}\simeq \D{P}(H^0(F(1)))$ identifies with the locus of tangents to $\D{P}(H^0(\cali_{l_F}(1)))\cap X^{*}$ i.e. with the pencils $\langle h, \D{T}_pX\rangle$ with $h$ an hyperplane containing $l_F$.)

Consider now the singular quadric surface $Q_p:=\D{T}_pX\cap X$. $Q_p$ is the cone with vertex $p$ over a smooth conic $C$ and the pencil $\D{P}(H^0(E(1)))$ uniquely determines 
a 1-dimensional linear system $\D{P}^1_E\subset \D{P}(H^0(\calo_C(1))$
hence, a point in the projective bundle
$$ \D{P}(\Sym^2(\calt^{\vee}))\longrightarrow \Grp(1,F(X))$$
 where, as usual, $\calt$ is the tautological bundle.
This projective bundle is a smooth 6 dimensional variety containing the variety $\calb\simeq \call(2)$ (see section \ref{sect-inst}) as an open subset. We can therefore always construct a regular affine curve with a marked point $(\spec(A), 0)$ and a 1-parameter family of pencils $\D{P}^1_t$, flat over $\spec(A)$,   such that $\D{P}^1_t\in \calb$ for $t$ general and whose central fibre $\D{P}^1_0$ is isomorphic to $\D{P}^1_E$.
Define $E_t:=\zeta^{-1}(\D{P}^1_t), \ t\ne 0$. This is a family of instantons, flat over $\spec(A)\setminus \{0\}$ such that the pencil of curves $\D{P}(H^0(E_t(1)))$ coincides with $\D{P}^1_t.$
$E_t$ admits a flat limit $E_0\in \overline{\call(2)}$ and since for such a sheaf, the pencil of curves $\D{P}(H^0(E_0(1)))$ must be the flat limit of $\D{P}^1_t$, we conclude that $E_0\simeq E$.
This means that $\psi(\tilde{\Sigma}_T)\subset \overline{\call(2)}\cap \calc$ therefore the same holds for $\calp$. 
\end{proof}

This shows the connectedness of $\calm$ ending the proof of Theorem \ref{complete}.


\end{document}